\documentclass[12pt]{amsart}
\usepackage{amssymb,latexsym}

\newcommand{\lab}[1]{\label{#1}}

\newcommand{\bi}{b}
\newcommand{\di}{d}

\newcommand{\zero}{{\bf{0}}}
\newcommand{\one}{{\bf{1}}}

\begin{document}

  \title[]{On the equational theory of  finite modular lattices}

  \author[C.~Herrmann]{Christian Herrmann}
\begin{abstract}
It is shown that there is $N$ such that there 
is no algorithm to decide for  identities in at most $N$  variables  
validity in the class of  finite modular lattices.
This is based on Slobodskoi's result that the Restricted
Word Problem is unsolvable for the class of finite groups
and relies on Freese's technique 
of capturing  group presentations within free modular lattices.
\end{abstract}

\maketitle

\newtheorem{thm}{Theorem}[section]
\newtheorem{pro}[thm]{Proposition}
\newtheorem{cor}[thm]{Corollary}
\newtheorem{lem}[thm]{Lemma}
\newtheorem{prob}[thm]{Problem}
\newtheorem{obs}[thm]{Observation}
\newtheorem{fact}[thm]{Fact}
\newtheorem{claim}[thm]{Claim}

\newcommand{\ov}[1]{\overline{#1}}

\newcommand{\mb}[1]{\mathbb{#1}}
\newcommand{\mc}[1]{\mathcal{#1}}
\newcommand{\co}[1]{}

\section{Introduction}
Since Dedekind's
 early result on modular lattices with $3$ generators,
calculations in modular lattices have served to reveal structure,
particularly  
in  algebraic and geometric contexts. 
 Though, as shown by Hutchinson \cite{hutch}
and Lipshitz \cite{lip}, the Restricted  Word Problem for modular lattices
is unsolvable ($5$ generators suffice)
and so is the Triviality Problem.
 These results remain valid for any class 
of modular lattices containing 
the  subspace lattices of some infinite dimensional 
vector space. 
Applying von Neumann's  rings 
associated with frames, i.e. coordinate systems,  
 in modular lattices, the proof
relies on interpreting a  finitely presented group
with unsolvable word problem cf.  Section~\ref{5.4}, below.

On the other hand,  for many
 rings $R$, including all  division rings and
homomorphic images of $\mathbb{Z}$, the equational theory of
 the class of all lattices $L(_RM)$ of submodules
of $R$-modules
is decidable \cite{HHwp}; a thorough
analysis has been given by   G\'abor Cz\'edli and George
Hutchinson
\cite{HC}.

Again based on frames and the fact, shown by
Andr\'as Huhn \cite{huhn},
that frames freely generate projective modular lattices,
Ralph Freese \cite{fr} proved unsolvability of the Word 
Problem for the modular lattice $FM(5)$ with $5$ free generators.
 On the model side, he used 
a construction, due to Dilworth and Hall,
obtaining a modular lattice matching an upper section of one with 
a lower section of the other - here applied to height $2$ 
intervals in subspace lattices of $4$-dimensional vector spaces.
The matching was such 
that  results of Cohn and McIntyre could be used
to 
 capture group presentations within one of the
associated skew fields.  
 On the syntactic side, this
structure is reflected  in a  sublattice
 of $FM(5)$, 
by a method to be called \emph{Freese's technique} (definitions are given below):
\begin {itemize}
\item Start with terms providing, in a free modular lattice, a configuration
  composed of frames.
\item Use ``reduction'' of frames (mimicking subquotients of  modules) to 
force additional relations, e.g. characteristic $p$
of  frames for some prime $p$.
\item Based on  integers            in coordinate rings,
construct  elements behaving well under reduction
(called ``stable'' in \cite{fm4}), a property  which is inherited 
under reduction and change of the reference frame via glueing. 
\item Use  stable elements obtained via glueing 
 as group generators and
force group relations on these via reduction. 
\end{itemize}
Actually, Freese's proof associates a projective modular
lattice with any finitely presented group
and his result  remains valid for all
varieties of modular lattices containing all infinite  modular
lattices of height $6$.

The case   of the free modular lattice with $4$ 
generators has been  done  in \cite{fm4}
 interpreting finitely presented  $2$-generator groups $G$
via a concept of ``skew-frames of characteristic $p\times p$'',
providing $2$ stable elements. Here,  models are  obtained by a glueing construction
involving a lattice ordered system of components 
given as lattices of submodules
of free  
  $(\mathbb{Z}/p\mathbb{Z})G$-modules.

The result crucial for the present note is the following.
\begin{thm}\lab{slo}
{\rm Slobodskoi} \cite{slo}.
The Restricted Word Problem for the class of  finite groups is
unsolvable.
That is, there is a list $\bar g=(g_1, \ldots ,g_n)$ 
of generator symbols and a finite set of 
   relations $\rho_i(\bar g)$  in the language of groups
such that there is no algorithm to decide,  for any   word $w(\bar g)$, 
whether $w(\bar a)=e$ for all finite groups $G$
and all $\bar a$ in $G$ satisfying the relations $\rho_i(\bar a)$
for all $i$.
\end{thm}
Kharlampovich \cite{olga} proved the analogue 
for finite nilpotent groups, even more restricted classes
of finite groups have been dealt with in \cite{olga2}. 
 A concise review
of Slobodkoi's result
has been given in \cite[Section 2]{brid}. 
For a detailed analysis  see \cite[Section 7.4]{sapir}.

A rather immediate consequence of Thm.~\ref{slo}
is  that the Restricted Word Problem
is unsolvable for any class  of finite modular lattices
containing all subspace lattices of finite vector spaces
 cf. Section~\ref{5.4}.
The same applies  to the Triviality Problem \cite{cons},
based on the unsolvability for the class of finite groups,  
proved by Bridson and Wilton  \cite{brid}.

\begin{thm}\lab{thm}
With $n$ from Slobodkoi's result, the set of   identities in  $n+6$ variables,
  valid in  all  finite   modular lattices, is non-recursive.
\end{thm}

 Thm~\ref{thm} adapts  to all classes of finite
modular lattices containing the particular ones
constructed  in Thm.~\ref{const}  from   groups in a class
with unsolvable restricted word problem.

In a recent related result, K\"uhne and Yashfe \cite{mat}
show  that there is no  algorithm 
to decide, for any  finite geometric lattice
$L$ with dimension function $\delta$, whether 
there is a join            embedding  
$\varepsilon$ of $L$ into the subspace lattice of some vector
space
(over fields from any specified class)
such that, for some          $c\in \mathbb{N}$,
 $\dim \varepsilon(x)= c \cdot\delta(x)$ for all $x\in L$. \\

Concerning the proof of Thm.~\ref{thm},
on the model side, the construction given in \cite[Section 4]{fm4} 
is extended combining $n$ skew frames
 (into a ``tower'') to deal with $n$  group generators (Section~\ref{6}).
On the formal side, for convenience,  we first capture towers 
by  presentations projective within modular lattices (Section~\ref{3}).
 Reductions of towers and stable elements  are discussed in Section~\ref{4},
coordinates and Freese's method of forcing group relations in 
Section~\ref{5}.
Finally, in Section~\ref{7}
 reduction according to \cite[Lemma 9]{fm4}
 is used to turn each 
  skew frame of a tower, one at a time,  into one of characteristic
$p\times p$ and to obtain the stable element provided by \cite[Cor. 13]{fm4}.  
Stable elements associated with other skew frames
 will be transformed
(thanks to Ralph Freese) into stable elements, again, 
 and it does not matter that  characteristic $p\times p$ 
is (supposedly) lost.
 For easy reference,
a  description of the (well known) 
general  method is given in Section~\ref{2}.

\section{Presentations and 
reduction  to identities}\lab{2}

\subsection{Presentations}
Given a similarity type  of algebraic structures,
we fix a variety $\mc{T}$  
with solvable word problem for free algebras
$F\mc{T}(\bar x)$ in 
free generators $\bar x=(x_1,\ldots ,x_n)$,
called \emph{variables}.
Elements $t(\bar x)$
will be called \emph{terms}.
For binary operator symbols, say $+$,
we write $s+\bar t$ to denote the list $(s+t_1, \ldots ,s+t_n)$. 

Due to the solvability of the word problem there
is an algorithm to decide, for any terms $t(\bar x),
s(\bar x)$ in the absolutely free algebra,
whether the \emph{identity}  $t(\bar x)=s(\bar x)$
is valid in $\mc{T}$, i.e.
whether the terms denote the same
element of $F\mc{T}(\bar x)$.
 This applies also to expansions by new constants.

To simplify notation,
 a list   $(t_1(\bar x), \ldots, t_n(\bar x))$ 
is also written as $\bar t(\bar x)$ 
and $\bar t(\bar x)|_m$  stands for
$(t_1(\bar x), \ldots t_m(\bar x))$ where $m \leq n$.
Also, we write $\bar 
t(\bar u(\bar y))=\bar t(u_1(\bar y), ,\ldots ,u_n(\bar y))$
and the like.

Given 
a list $\bar c$ of (pairwise distinct) new constants, called
\emph{generator symbols}, a \emph{relation}  $\rho(\bar c)$
is an expression $t(\bar c)=s(\bar c )$
where $t(\bar x)$ and $s(\bar x)$ 
are terms.
A  (finite) \emph{presentation}  $\Pi$ (also written as $(\Pi,\bar c)$)
is then given by $\bar c$ and a finite
set of relations $\rho(\bar c)$.
Constant (i.e. variable free) terms 
in the language expanded by $\bar c$ are also referred to as
``terms over $\Pi$''.

A relation $\rho(\bar c)$, as above, is \emph{satisfied} 
by $\bar a=(a_1,\ldots, a_n)$
in  $A\in \mc{T}$ if 
$t(\bar a)=s(\bar a)$, we write $A\models \rho(\bar a)$.
 $(A,\bar a)$ 
is a \emph{model} of $\Pi$, written as $A\models \Pi$, 
if all relations of $\Pi$ are satisfied in $A$.
Par abuse de  language we also say that $\bar a$ is a $\Pi$ in $A$
and we use $\bar c$ to denote $\bar a$.

In the sequel, let $\mc{A}\subseteq \mc{T}$ denote
a class of algebraic structures
closed under subalgebras.
 We say that $(A,\bar a)$ \emph{is
in} $\mc{A}$ if $A\in \mc{A}$. 
A relation $\rho(\bar c)$
is a \emph{consequence} of $\Pi$ in $\mc{A}$, also
 \emph{implied} by $\Pi$ in $\mc{A}$,
if $A\models \rho(\bar a)$ 
for all models $(A,\bar a)$ of $\Pi$ in $\mc{A}$. 
The \emph{Restricted Word Problem}
for $\mc{A}$ is \emph{unsolvable}
if there is a presentation $(\Pi,\bar c)$
such there is no algorithm
to decide, for any relation $\rho(\bar c)$, 
whether $\rho(\bar c)$ is a consequence of $(\Pi,\bar c)$ within 
$\mc{A}$.

\subsection{Transformations and strengthening}
A  \emph{transformation} within $\mc{A}$ 
of $\Pi$ to a presentation $\Psi$ 
in  generator symbols $\bar d=(d_1, \ldots ,d_m)$
  is given by 
a list of terms $u_j(\bar x)$, $j=1, \ldots, m$,
such that  one has $(A,(u_1(\bar a), \ldots, u_m(\bar a))$
a model of $\Psi$ 
for each model $(A,\bar a)$ of $\Pi$ in $\mc{A}$.
The \emph{composition} with a further transformation $\Psi$ to $\Phi$,
given by the $v_k(\bar y)$, is the transformation  obtained by the terms
$v_k(u_1(\bar x), \ldots ,u_m(\bar x))$. 
Thus, one  obtains transformations by iterated composition.
The presentations $\Pi$ and $\Psi$
are \emph{equivalent} within $\mc{A}$ if, in the above,
one has $\Phi=\Pi$ and
  $A\models \bar v(\bar u(\bar a))=\bar a$,
and $B \models \bar u(\bar v(\bar b))=\bar b$ 
for all models $(A,\bar a)$ of $\Pi$ and $(B,\bar b)$ of $\Psi$.
In particular, if $\Psi$ is obtained from $\Phi$ 
adding generators (that is, $m>n$ and
$c_i=d_i$ for $i\leq n$) and relations then $\Pi$ and $\Psi$
are equivalent within $\mc{A}$ if and only if
there is a transformation of $\Phi$ to $\Psi$
within $\mc{A}$ such that  $u_j=x_j$ for $j \leq n$.

Consider presentations $\Pi$ and $\Pi^+$ in the same generator
symbols $\bar c=(c_1,\ldots ,c_n)$.
A transformation from $\Pi$ to $\Pi^+$ within $\mc{A}$
 given by  terms $u_i(\bar x)$, $i=1,\ldots ,n$
(also written as $u_{c_i}(\bar x)$  with $x_i=x_{c_i}$)  \emph{strengthens} $\Pi$ to $\Pi^+$ 
within $\mc{A}$ if the following hold.
\begin{enumerate}
\item
The relations of $\Pi$ are consequences of $\Pi^+$ within
$\mc{A}$.
\item
 $u_i(\bar a)=a_i$ for $i=1,\ldots ,n$ 
and all models $(A,\bar a)$ of $\Pi^+$ in $ \mc{A}$.
\end{enumerate}
That is, from  any model $(A,\bar a)$ of $\Pi$ in $\mc{A}$
one obtains the model $(A,\bar u(\bar a))$
of $\Pi^+$ while models of $\Pi^+$ remain unchanged.

Considering a model $(A,\bar a)$ of $\Pi$,
it is common use to write also $c_i$ in place of $a_i$,
that is, the generator symbol $c_i$ denotes
the element $a_i$ of $A$. In view of this, 
we use the notation
$c_i:= u_i(\bar c)$
to indicate the terms $u_i(\bar x)$ defining the
strengthening of $\Pi$ 
 to $\Pi^+$ -
without  mentioning $c_i:=c_i$ if $u_i(\bar x)=x_i$.
 In particular
this is done if we construct a sequence of strengthenings 
- which, of course, provides a strengthening of the original
presentation.

\subsection{Projective  presentations}
A presentation 
$\Pi$ is \emph{projective} within  $\mc{A}$ if there are 
(\emph{witnessing}) terms
$t_1^\Pi(\bar x), \ldots , t_n^\Pi(\bar x)$
such that the following hold for all $A\in \mc{A}$
and $\bar a$ in $A$.
\begin{enumerate}
\item  $(A,(t_1^\Pi(\bar a), \ldots , t_n^\Pi(\bar a))\models \Pi$.
\item If $(A,\bar a)\models \Pi$ 
then $t_i^\Pi(\bar a)=a_i$ for all $i$.
\end{enumerate}
Then, of course, $\Pi$ is projective within any $\mc{B} \subseteq \mc{A}$.

\begin{fact}\lab{newgen}
If $\Pi_1$ and $\Pi_2$  are  projective within $\mc{A}$
then so is their disjoint union, e.g. 
if $\Pi_2$ introduces  additional generators, but no relations.
\end{fact}

\begin{fact}
If $\Pi$ is  strengthened to $\Pi^+$
within $\mc{A}$ then $\Pi^+$ is  projective
 in $\mc{A}$ 
 if so is $\Pi$.
\end{fact}

\begin{fact}
If $\Pi$ is projective in  $\mc{A}$,
with witnessing terms $t_i^\Pi(\bar x)$, then 
the identity \[t(t_1^\Pi(\bar x), \ldots ,t_n^\Pi(\bar x))
=s(t_1^\Pi(\bar x), \ldots ,t_n^\Pi(\bar x))\] 
is valid in $\mc{A}$ if and only if
$t(\bar a)= s(\bar a)$ for all models 
$(A,\bar a)$ of $\Pi$ in  $\mc{A}$.
\end{fact}

Now, assume that $\mc{A}$ is a variety, i.e. an equationally definable
class. 
 Then for each presentation $(\Pi,\bar c)$ 
one has ``the''  algebra $F\mc{A}(\Pi,\bar c)$
in $\mc{A}$ freely generated  
by $\bar c$ 
  under the relations $\Pi$;
here,  $\bar c$ 
also denotes its image under the canonical homomorphism.
This algebra is projective within $\mc{A}$ if and only
if so is the presentation $(\Pi,\bar c)$.

Strengthening the  presentation $(\Pi,\bar c)$ 
to $\Pi^+$ with additional relation $s(\bar c)=t(\bar c)$
 then means
to provide $\bar b$ in 
$F\mc{A}(\Pi,\bar c)$ such that $s(\bar b) =t(\bar b)$ 
and $\phi(\bar b)=\phi(\bar c)$
for all
  $A \in \mc{A}$ and homomorphisms 
$\phi:F\mc{A}(\Pi,\bar c)\to A$  
such that $s(\phi(\bar c))=t(\phi(\bar c))$.

\subsection{Reducing quasi-identities to identities}\lab{2.6}
Given a signature, an \emph{identity} or \emph{equation}
is a sentence of the form $\forall \bar x.\; t(\bar x)=s(\bar x)$,
a \emph{quasi-identity} a sentence of the form
$\forall \bar y.\; \alpha(\bar y) \Rightarrow t(\bar y)=s(\bar y)$
with \emph{antecedent}
$\alpha(\bar y)\equiv \bigwedge_i t_i(\bar y)=s_i(\bar y)$;
here $t(\bar x)$, $s(\bar x)$, $t_i(\bar y)$, and $s_i(\bar y)$ are terms.
Observe that, replacing variables by new constants,
$\alpha$ is turned into a presentation.

Consider classes $\mc{M}_0$ and $\mc{G}_0$ 
of algebraic structures in not necessarily distinct 
signatures, both closed under subalgebras.
The task is to \emph{reduce} quasi-identities for $\mc{G}_0$
to equations for $\mc{M}_0$; that is,
given  a set $\Lambda$ of
quasi-identities  in the language of $\mc{G}_0$
 to 
construct an algorithm associating with each 
$\beta \in\Lambda$
 an equation $\beta^*$ in the language of $\mc{M}_0$
such that $\beta$ holds in $\mc{G}_0$ if and only 
if $\beta^*$ holds in $\mc{M}_0$.

In the sequel we describe the general structure of such
algorithm to be applied to the case where $\mc{G}_0$
is the class of all finite groups, $\mc{M}_0$ the class
of all finite modular lattices.
Fix a set $\Lambda_0$
of formulas 
 $\alpha(\bar y)\equiv  \bigwedge_{i=1}^h w_i(\bar y)= v_i(\bar y)$,  $\bar y=(y_1, \ldots ,y_{n_\alpha})$, 
in the language of $\mc{G}_0$.

{\bf Hypothesis}: 
 There is an algorithm
which constructs the following in the language of $\mc{M}_0$.
\begin{itemize}
\item[(a)]
For any  given $\alpha \in \Lambda_0$
a  presentation $(\Pi,\bar c)$ with 
$\bar c=(c_1,\ldots ,c_N)$ and
terms $\bar u=(u_1,\ldots, u_N)$ with $N:=N_\alpha \geq n:= n_\alpha$
\item[(b)]
 For each $r$-ary  operation symbol $f$ 
of $\mc{G}_0$,   a term $f^\#(\bar z, \bar x)$, $\bar z=(z_1,\ldots ,z_r)$
where $\bar x=(x_1, \ldots ,x_N)$.
\end{itemize}
Now, for a formula $\gamma(\bar y)$ in the language of $\mc{G}_0$, 
the translation according to (b)  into a formula 
in the language of $\mc{M}_0$ is denoted by $\gamma^\#(\bar y,\bar x)$
and the following are required.
\begin{itemize}
\item[(i)] $(\Pi,\bar c)$ is projective for $\mc{M}_0$
with witnessing terms $\bar t$.
\item[(ii)] If $(L,\bar a)$ is a model of $(\Pi,\bar c)$ 
in $\mc{M}_0$ then so is $(L,\bar u(\bar a))$.
\item[(iii)] For any $\alpha \in \Lambda_0$,
 $(\Pi,\bar c)$ implies $\alpha^\#(\bar u(\bar c)|_n),\bar u(\bar c))$ in $\mc{M}_0$.
\item[(iv)] For any model $(L,\bar a)$ of $(\Pi,\bar c)$
with $L\in \mc{M}_0$, the algebra $G=G(L,\bar u|_n(\bar a))$ generated by 
$\bar u|_n(\bar a)$
under the operations $\bar b \mapsto f^\#(\bar b,\bar u|_n (\bar a))$,
$f$ an operation symbol of $\mc{G}_0$, is a member of $\mc{G}_0$ 
and $(G,\bar u|_n(\bar a))\models \alpha(\bar u|_n(\bar a))$.
\item[(v)] For any $\alpha\in \Lambda_0$ and $G\in \mc{G}_0$, 
with generators $\bar g=(g_1, \ldots ,g_n)$ 
such that $G\models \alpha(\bar g)$,
there is a model  $(L(G,\bar g),\bar a)$ of $(\Pi,\bar c)$
with $L(G,\bar g)  \in \mc{M}_0$ 
and $\bar u(\bar a) =\bar a$ and, moreover, such that 
there is an embedding $\omega:G \to G(L(G,\bar g),\bar a)$  
with $\omega(g_i)=a_i$ for $i=1,\ldots, N$.
\end{itemize}

\begin{lem}\lab{reduction}
Given an algorithm satisfying the above hypothesis, there is
an algorithm associating with any quasi-identity $\beta$,
with antecedent $\alpha\in \Lambda_0$
in the language of $\mc{G}_0$, 
an equation $\beta^*$ in the language of $\mc{M}_0$
such that $\beta$ holds in $\mc{G}_0$ 
if and only if $\beta^*$ holds in $\mc{M}_0$.

In particular, if a presentation $\Psi$ in $n$ generators   in the language of
$\mc{G}_0$ is given, then the restricted word problem 
for $\Psi$ within $\mc{G}_0$ reduces
to the decision problem for $N$-variable identities within $\mc{M}_0$
where  $N$ is the number of generators
in the presentation $(\Pi,\bar c)$,  required in (a) above.
\end{lem}
\begin{proof}
By (b), there is an algorithm associating, uniformly for all $n,N$
$(n \leq N$),
 with any term
$w(\bar y)$ in the language of $\mc{G}_0$  a term $w^\#(\bar y,\bar x)$
in the language of $\mc{M}_0$ such that $y_i^\#(\bar y,\bar x) =y_i$
and \[(f(w_1(\bar y), \ldots,w_n(\bar y)))^\#(\bar y,\bar x)
=f^\#(w_1^\#(\bar y,\bar x), \ldots ,w_n^\#(\bar y,\bar x)).\]  
Now, 
given $\beta \equiv \forall y.(\alpha(\bar y) \Rightarrow w(\bar y)=v(\bar y))$
where $\alpha \in \Lambda_0$, 
 let $\beta^*$ denote the identity $\forall \bar x.\;
\gamma(\bar x)$ where $\gamma(\bar x)$ denotes 
\[
w^\#(u_1(\bar t(\bar x)), \ldots , u_n(\bar t(\bar x)), \bar u(\bar t(\bar x))  )= \]\[
=v^\#(u_1(\bar t(\bar x)), \ldots , u_n(\bar t(\bar x)), 
\bar u(\bar t(\bar x))  )\]
with $u_i(\bar x)$ according to (a).
Assume that $\beta^*$ holds in $\mc{M}_0$ and consider
$G\in \mc{G}_0$ and $\bar g$ in $G$ such that $G\models \alpha(\bar g)$.
Given  $(L(G,\bar g), \bar a)$ according to (v), one has
 $u_i(\bar a)=a_i$ for all $i$ whence, due to validity of $\beta^*$,
\[ \omega(w(\bar g))=w^\#(\bar a|_n,\bar a)=
   w^\#(\bar u(\bar a)|_n), \bar u(\bar a))=\]\[ = v^\#(\bar u(\bar a)|_n,\bar u(\bar a)) 
=v^\#(\bar a|_n,\bar a))= 
\omega(v(\bar g))      \]
and  $w(\bar g)=v(\bar g)$ follows, verifying $\beta$ for $\mc{G}_0$.

Conversely, 
assume that $\beta$ holds for all $G$ in $\mc{G}_0$ 
and consider any $L \in \mc{M}_0$ and $\bar a=(a_1, \ldots ,a_N)$
in $L$. That $(L,\bar u(\bar t(\bar a)))$ is a model of 
$(\Pi, \bar c)$ is obtained combining (i) and (ii)
and by (iii) it follows that $L\models \alpha^\#(\bar u(\bar t(\bar a))|_n,
\bar u(\bar t(\bar a)))$.  
Thus, by  (iv) $G:=G(L, \bar u(\bar t(\bar a))|_n)$ 
is in $\mc{G}_0$ and  
 $\alpha(\bar u(\bar t(\bar a))|_n)$ 
holds in $G$.
Now, the quasi-identity $\beta$
being valid in $\mc{G}_0$, 
it follows  
 $G\models \gamma(\bar u(\bar t (\bar a))|_n, 
\bar u(\bar t(\bar a)))$; that is, the identity $\beta^*$ 
holds in $L$ for the substitution $\bar a$..
\end{proof}

\section{Some projective  modular lattice presentations}\lab{3}

\subsection{Terms and lattices}
For concepts of lattice theory we refer to Birkhoff 
\cite{birk}, for modular lattices also von Neumann \cite{neu}. For better readability, joins and meets will
be written as $x+y$ and $x\cdot y=xy$,
 assuming associativity, commutativity. and idempotency
for both operations. 
That is, the term algebra $F\mc{T}(\bar x)$
is the free algebra in the variety $\mc{T}$ 
of algebras $(A,+, \cdot)$ where $(A,+)$ and $(A,\cdot)$ are
commutative idempotent monoids.
Thus, the word problem for free algebras in $\mc{T}$
has a (simple) solution
and we may use expressions   $\sum_i a_i$ and $\prod_i a_i$
- to be read as $(\sum_i a_i)$ and $(\prod_i a_i)$, respectively.
 For convenience, we also  use the 
rule that  $s\cdot t +u=st+u$ reads as  $(st)+u$.

A \emph{lattice} $L$ is a member of $\mc{T}$
which satisfies the \emph{absorption laws}
\[x(x+y)=y \mbox{ and } x+xy=x.\] For lattices, $x \leq y \Leftrightarrow
x=xy$ (we also write $y\geq x$) 
defines a partial order $\leq$ and one has $a \leq b$
if and only if $a=a+b$. 
 With respect to this partial order, $a+b$ is the supremum,
$ab$ the infimum of $a,b$. 
If $L$ has a smallest resp. greatest element these 
will be denoted by $\bot^L$ and $\top^L$, respectively.
A set  $\{a_1,\ldots ,a_n\}$ in $L$ such that $\sum_i a_i =\top^L$ and
$\prod_i a_i =\bot^L$ will be called \emph{spanning} in $L$.
In particular, this applies if $L$ is generated by  $a_1, \ldots ,a_n$. 
For $a\leq b$ in $L$, the \emph{interval} $[a,b]=\{c\in L\mid a\leq c\leq b]$
is a sublattice of $L$; an \emph{ideal} 
of $L$ is a sublattice $I$ such that $b\in I$ for any $b \leq a \in I$. 
 The word problem for free lattices is
well known to be solvable, but for simplicity 
 we prefer to consider terms
in $\mc{T}$.

A \emph{chain} $C_n$ of length $n$
is a presentation with generators $d_i$, $i=0, \ldots, n$,
and relations $d_i \leq d_{i+1}$, $i <n$.
Obviously, chains are projective within the class of all lattices.

\subsection{Modular lattices}
A lattice is \emph{modular}
if it satisfies the identity $x(y+xz)= xy+xz$,
equivalently, if $a(b+c)=ab+c$ for all $c\leq a$.
The class of all modular lattices is denoted by $\mc{M}$.
 \emph{Projectivity} of
presentations will always refer to $\mc{M}$.
Examples of modular lattices are the lattices $L(_RM)$  of all submodules
of $R$-modules, with operations $+$ and $\cap$.

\begin{fact}\lab{mod}
In a modular lattice, $x \mapsto x+b$ is an
isomorphism of $[ab,a]$ onto $[b,a+b]$ with inverse 
$y\mapsto ya$. 
\end{fact}
Accordingly, we define  $\bar x \nearrow \bar y$     
to stand for the  formula
\[ \bigwedge_{i=1}^n \bigl( y_i =x_i + \prod_{j=1}^n y_j \quad
 \wedge \quad  x_i = y_i \cdot  \sum_{j=1}^n x_j \bigr).
  \] 
Observe that for $x_1\geq x_2$ and $y_1\geq y_2$ 
one has $x_1,x_2 \nearrow y_1,y_2 $ if and only if $x_1 \leq y_1$, $x_2 \leq y_2$,
and $y_2, x_2 \nearrow y_1,x_1$.
Also, $\bar x \nearrow \bar y$ and $\bar y \nearrow \bar z$
jointly imply $\bar x \nearrow \bar z$.
Moreover,,  $\bar x \nearrow \bar y$, $\bar x \nearrow \bar z$
and $\bar y \leq \bar z$ jointly  imply 
$\bar y \nearrow \bar z$.
Writing $\bar x^1 \nearrow \ldots  \nearrow \bar x^m$
we require $\bar x^i \nearrow \bar x^j$ for all $i<j \leq m$.\\

Call elements $a_1, \ldots ,a_n$ of a modular lattices \emph{relatively}
\emph{independent} (\emph{over} $b$)
if $b=a_k\cdot \sum_{i<k} a_i $ for all $1<k\leq n$.
This implies that any permutation of $a_1,\ldots ,a_n$
is independent over $b$, too, and that the $a_i$ 
generate a boolean sublattice $B$ with smallest element $b$ 
and each  $a_i$ is either $b$  or an atoms of  $B$.

\begin{fact}\lab{mod2}
In a modular lattice, if $u,v,w$ are relatively  independent over $t$,
then $(x,y,z) \mapsto x+y+z$ 
defines an embedding of  $[t,u] \times [t,v] \times [t,w]$ 
into $[t,u+v+w]$.
In particular, for $t\leq u'\leq u''\leq u$,
 $t \leq v'\leq v'' \leq v$, and $t \leq w'\leq w'' \leq w$
 the sublattice generated by 
these  $3$ chains
is isomorphic to the direct  product of these chains.
and the above
 embedding restricts to  isomorphisms
$x \mapsto x+v'$ of $[u'+w',u''+w'']$  onto $[u'+v'+w',u''+v'+w'']$ 
 and $y \mapsto y+u'$ of $[v'+w',v''+w'']$ onto $[u'+v'+w', u'+v''+w'']$,
respectively.
\end{fact}
See Fig.~\ref{prodfig}.
For the proof observe, that one can assume $t=0$
and that the case $w=0$ is well known.
The analogous result holds for any number of relatively 
independent elements.

\begin{figure} 
\setlength{\unitlength}{7mm}
\begin{picture}(5,9)(-2,-5)

\put(0,0){\line(-2,1){2}}
\put(0,0){\line(2,-1){6}}
\put(0,1){\line(2,-1){6}}
\put(1,1){\line(2,-1){6}}
\put(1,2){\line(2,-1){6}}

\put(-2,1){\line(-1,-1){3}}
\put(0,1){\line(-2,1){2}}
\put(0,0){\line(2,-1){2}}
\put(0,1){\line(2,-1){2}}
\put(1,1){\line(2,-1){2}}
\put(1,2){\line(2,-1){2}}
\put(-2,2){\line(-1,-1){3}}
\put(0,1){\line(-1,-1){3}}
\put(1,1){\line(-2,1){2}}
\put(1,2){\line(-2,1){2}}
\put(0,0){\line(1,1){1}}
\put(-2,1){\line(1,1){1}}
\put(-2,2){\line(1,1){1}}
\put(0,1){\line(1,1){1}} 
\put(0,0){\line(0,1){1}}
\put(1,1){\line(0,1){1}}
\put(-2,1){\line(0,1){1}}
\put(-1,2){\line(0,1){1}}

\put(-3,-3){\line(0,1){2}}
\put(-3,-3){\line(1,1){.5}}
\put(-5,-2){\line(1,1){.5}}
\put(-3,-2){\line(1,1){.5}}
\put(-5,-1){\line(1,1){.5}}

\put(-3,-3){\line(-2,1){4}}
\put(-3,-3){\line(0,-1){1}}

\put(-3,-2){\line(-2,1){4}}
\put(-3,-3){\line(2,-1){2}}
\put(-3,-2){\line(2,-1){2}}
\put(-5,-2){\line(0,1){2}}
\put(-5,-2){\line(0,-1){1}}
\put(-1,-5){\line(-2,1){6}}
\put(-1,-2){\line(-2,1){6}}
\put(-1,-2){\line(0,-1){3}}
\put(-7,-2){\line(0,1){3}}

\put(6,-3){\line(0,1){2}}
\put(6,-3){\line(0,-1){1}}
\put(6,-3){\line(1,1){2}}
\put(6,-3){\line(-1,-1){1}}
\put(5,-2){\line(1,1){3}}
\put(5,-3){\line(1,1){3}}
\put(5,-5){\line(1,1){3}}
\put(5,-5){\line(0,1){3}}

\put(7,-3){\line(0,1){3}}
\put(8,-2){\line(0,1){3}}

\put(0,0){\line(2,-1){6}}
\put(0,0){\line(-1,-1){3}}

\put(-.8,-5.1){$t$}
\put(-.8,-4.1){$w'$}
\put(-.8,-3.1){$w''$}
\put(-.8,-2.1){$w$}

\put(4.5,-5.1){$t$}
\put(4.3,-4.1){$w'$}
\put(4.3,-3.1){$w''$}
\put(4.5,-2.1){$w$}

\put(-5.7,-3.4){$u''$}

\put(-7.7,-2.4){$u$}
\put(-3.7,-4.4){$u'$}

\put(8.2,-2){$v$}
\put(7.2,-3.3){$v''$}
\put(6.2,-4.3){$v'$}

\put(0.2,0){$u'+v'+w'$}

\put(-.8,2.9){$u''+v''+w''$}

\put(-5.2,1.9){$u''+v'+w''$}
\put(1.2,1.9){$u'+v''+w''$}

\end{picture}
\caption{Direct product of chains}\label{prodfig}
\end{figure}
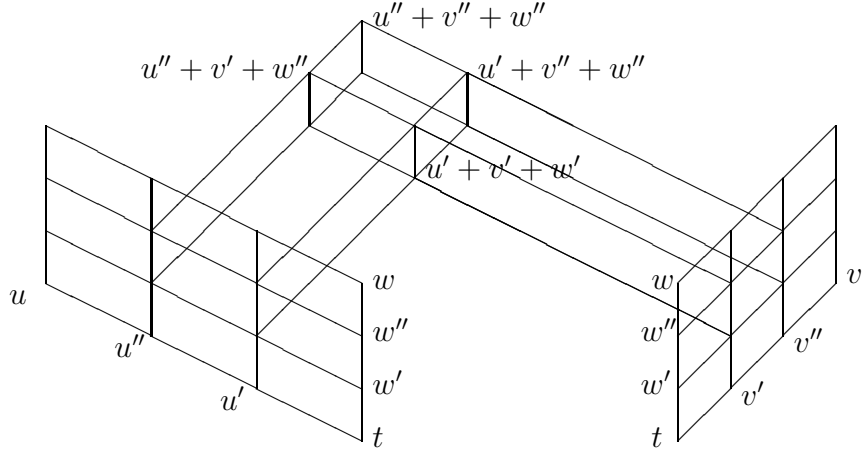.

\subsection{Products of presentations}
Given presentations $(\Pi^j,( \bot,\bar c^j))$, where
$\bar c^j=(c^j_1, \ldots ,c^j_{n_j})$ for $j=1,2$
with pairwise distinct $c_i^j$
and, for $j=1,2$, relations in $\Pi^j$ implying $\bot\leq c_i^j$
for all $i$, 
the \emph{product} $(\Pi^1,(\bot,\bar c^1))\times (\Pi^2,(\bot,\bar c^2))$
is the presentation $(\Pi,(\bot,\bar c))$  with generator symbols 
\[\bot \mbox{ and }\bar c=(c_1^1, \ldots ,c_{n_1}^1, c_1^2, \ldots c_{n_2}^2)\]
and, in addition to the relations of the $(\Pi^j,(\bot,\bar c^j))$,
the relations
\[  \sum_{i=1}^{n_1} c_i^1 \cdot \sum_{k=1}^{n_2} c_k^2
=\bot.  \] 
Fact~\ref{mod2} implies the following well known fact.

\begin{fact}\lab{prod}
Within $\mc{M}$, products of projective presentations
are projective. Moreover, given models $(L_j,(a^j_0,\bar a^j))$
of $(\Pi^j,(\bot,\bar c^j))$ in $\mc{M}$
one has $L_1\times L_2$ a model of the product
with generators mapped to $(a^1_i,a^2_0)$ and $(a^1_0,a^2_k)$,
respectively;
moreover,  any model of the product of the  presentations is isomorphic to such.
\end{fact}

\subsection{Frames}
Frames have been introduced by von Neumann
\cite{neu} for coordinatizing complemented modular lattices.
Given $n$ independent generators  $e_i$
of an $R$-module $_RM$, the \emph{canonical}
$n$-frame in $L(_RM)$ consists of $a_i=Re_i$, $c_{1j}= R(e_1-e_j)$,
and $a_\bot=\{0\}$. This is mimicked by the following 
presentation.

An  $n$-\emph{frame} $\Phi$ is a lattice presentation
with generators  $a_\bot$,  $a_1, \ldots ,a_n$,
$c_{1j}=c_{j1}$ ($2\leq j\leq n$) and relations
\[\begin{array}{lrclrcl}
(1)&a_\bot &=&a_j(\sum_{i=1}^{j-1}a_i)\\ 
(2)&  a_\bot&=& a_1c_{1j} &=&a_jc_{1j}\\
(3)& a_1+a_j&=& a_1+c_{1j}&=&a_j+c_{1j}\\ \end{array}   \]
where  $2 \leq j \leq  n$. See Fig.~\ref{framefig}

An equivalent presentation is obtained
by replacing $a_\bot$ by $a_1a_2$.
A model (in $\mc{M}$) of (an) $n$-frame is referred to as ``an $n$-frame
in a modular lattice''; otherwise, speaking of ``an $n$-frame''
we mean a presentation as above, possibly with renamed generators.

\begin{figure}
\setlength{\unitlength}{10mm}
\begin{picture}(5,2)(-4,-1.7)

\put(2,-2){\circle*{0.2}}
\put(2,-2){\line(2,1){2}}
\put(2,-2){\line(-2,1){2}}
\put(2,-2){\line(1,1){2}}
\put(2,-2){\line(-1,1){2}}
\put(2,-1){\line(2,1){2}}
\put(2,-1){\line(-2,1){2}}
\put(2,-2){\line(0,1){1}}
\put(0,-1){\line(0,1){1}}
\put(4,-1){\line(0, 1){1}}

\put(1,-1){\circle*{0.2}}
\put(0,-1){\circle*{0.2}}
\put(2,-1){\circle*{0.2}}
\put(3,-1){\circle*{0.2}}
\put(4,-1){\circle*{0.2}}
\put(0,0){\circle*{0.2}}
\put(4,0){\circle*{0.2}}

\put(2.2,-2.2){$a_\bot$}
\put(-.5,-1){$a_2$}
\put(1.2,-1){$c_{12}$}
\put(2.2,-1.2){$a_1   $}
\put(3.2,-1){$c_{13}$}
\put(4.2,-1){$a_3   $}
\end{picture}

\begin{picture}(0,0)(0,0)
\put(-3.6,1){$a_2$}
\put(-1.8,0.9){$c_{12}$}
\put(-.8,0.9){$a_1$}
\put(-1.8,2.1){$a_\top$}
\put(-1.8,-0.1){$a_\bot$}

\put(-2,1){\circle*{0.2}}
\put(-3,1){\circle*{0.2}}
\put(-1,1){\circle*{0.2}}
\put(-2,0){\circle*{0.2}}
\put(-2,2){\circle*{0.2}}
\put(-2,0){\line(-1,1){1}}
\put(-2,0){\line(1,1){1}}
\put(-2,0){\line(0,1){1}}
\put(-3,1){\line(1,1){1}}
\put(-1,1){\line(-1,1){1}}
\put(-2,1){\line(0,1){1}}

\end{picture}
\caption{$2$-frame and $3$-frame}\label{framefig}
\end{figure}
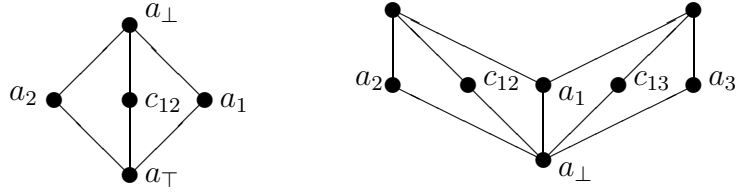

We define $a_\top=\sum_{i=1}^n a_i$ and 
 write also   $a^\top=a_\top^\Phi=\top^\Phi$, $a_i=a_i^\Phi$,  $c_{1j}=c_{1j}^\Phi$, and
$a_\bot=a_\bot^\Phi=\bot^\Phi$.  The list of generators
with indices not involving fixed $k>$ is written as $\Phi_{\neq k}$.
 Observe that  $\Phi$ implies,
within $\mc{M}$,  the relations of $n-1$-frames 
for $\Phi_{\neq k}$ and $a_k+\Phi_{\neq k}$ and that
\[  \Phi_{\neq k} \nearrow
  a_k + \Phi_{\neq k}.\]
Also observe, that  the concept of $n$-frame
can be defined, recursively: 
start  with that of $2$-frame, as defined above;
now,  given the concept  of $n$-frame $\Phi$,  obtain the  $n+1$-frame
$\Phi^+$ adding to the generators and relations of $\Phi$
the generators $a_{n+1}$ and $c_{1,n+1}$
and relations (2) and (3) for $j=n+1$ - renaming
$a_\bot^\Phi$ into $a_\bot^{\Phi^+}$. 
The following is a special case of Dedekind's 
description of $3$-generated modular lattices.

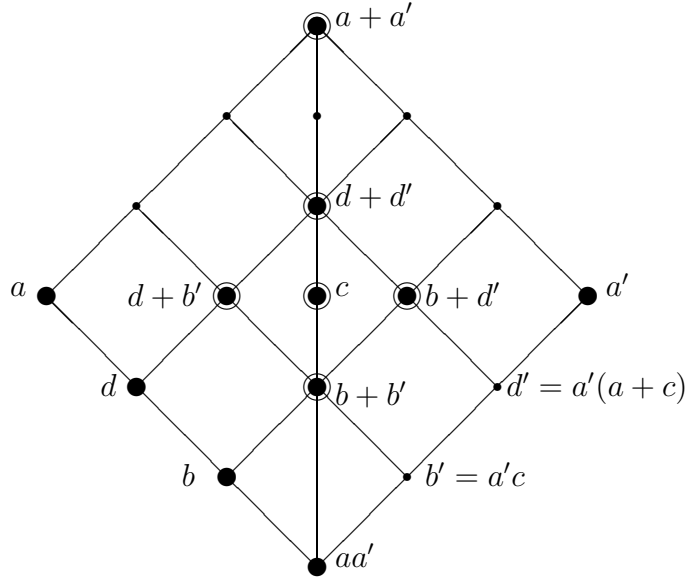
\begin{figure}
\setlength{\unitlength}{12mm}
\begin{picture}(0,6.5)(0,0)  
\put(0,0){\circle*{0.2}}
\put(0,6){\circle*{0.2}}
\put(-3,3){\circle*{0.2}}
\put(3,3){\circle*{0.2}}
\put(0,2){\circle*{0.2}}
\put(0,6){\circle*{0.2}}
\put(1,3){\circle*{0.2}}
\put(-1,3){\circle*{0.2}}
\put(0,4){\circle*{0.2}}
\put(0,3){\circle*{0.2}}
\put(0,2){\circle{0.3}}
\put(0,6){\circle{0.3}}
\put(1,3){\circle{0.3}}
\put(-1,3){\circle{0.3}}
\put(0,4){\circle{0.3}}
\put(0,3){\circle{0.3}}
\put(-2,2){\circle*{0.1}}

\put(1,1){\circle*{0.1}}
\put(2,2){\circle*{0.1}}
\put(-1,1){\circle*{0.2}}
\put(-2,2){\circle*{0.2}}
\put(-2,4){\circle*{0.1}}
\put(-1,5){\circle*{0.1}}
\put(2,4){\circle*{0.1}}
\put(1,5){\circle*{0.1}}
\put(0,0){\line(1,1){3}}
\put(0,0){\line(-1,1){3}}
\put(-1,1){\line(1,1){3}}
\put(1,1){\line(-1,1){3}}
\put(-2,2){\line(1,1){3}}
\put(2,2){\line(-1,1){3}}
\put(-3,3){\line(1,1){3}}
\put(3,3){\line(-1,1){3}}
\put(0,2){\line(0,1){2}}

\put(0.2,6){$a+a'$}
\put(0.2,4){$d+d'$}
\put(0.2,3){$c$}
\put(1.2,2.9){$b+d'$}
\put(-2.1,2.9){$d+b'$}
\put(3.2,3){$a'$}
\put(-3.4,3){$a$}
\put(.2,0){$aa'$}
\put(.2,1.8){$b+b' $}
\put(-1.5,.9){$b$}

\put(2.1,1.9){$d'=a'(a+c)$}
\put(1.2,.9){$b'=a'c$}
\put(-2.4,1.9){$d$}

\end{picture}

\caption{Modular lattice generated by 
$a,a',c$ with $aa'\leq c \leq a+a'$ }\label{dede}
\end{figure}.

\begin{fact}\lab{mod3}
The modular lattice freely generated by $a,a',c$ 
such that $aa' \leq c \leq a+a'$
has diagram given in Fig.~\ref{dede}.
In particular, with 
 $b=ac$ and  $d=a(a' +c)$
one has a $2$-frame $b+ a'c$, $d +a'c$, $b+a'(a+c)$, $c$
and $b,d  \nearrow b+a'c, d+a'c$.
\end{fact}

\subsection{Reduction}
Given an $n$-frame $\Phi$ and variables $x,y$, put 
\[a_\bot(x,y)\equiv x+\sum_{j>1} a_j(x+c_{1j})
  \mbox{ and } a_\top(x,y)\equiv y+\sum_{j>1} a_j(y+c_{1j})\]
and introduce for each remaining  generator symbol $c$ in $\Phi$ the term
\[\hat{c}(x,y) \equiv   c\cdot a_\top(x,y)+a_\bot(x,y).  \]
Observe that for all models of $\Phi$ 
in a modular lattice $L$  
and $b,d$ in $L$ with  $a_\bot\leq b\leq d \leq a_1$   one has the identity $\hat{c}(b,d)= (c+a_\bot(b,d))\cdot a_\top(b,d)$.
Let $\Phi(x,y)=\Phi_x^y$ denote the list of terms
 $\hat{c}(x,y)$, $c$ a generator symbol of $\Phi$;
this is called the \emph{reduction setup} for $n$-frames.

If $\Phi$ is part of a presentation $\Pi$ and
$b,d$ are terms over $\Pi$ then 
$\Phi_d^b$ is obtained  substituting $b,d$  for $x,y$;
 $\Phi_d^b$ is called the \emph{reduction} of $\Phi$ via $b,d$.
We put $\Phi^b= \Phi^b_{a_\bot}$ and $\Phi_d= \Phi^{a_1}_d$.

If $B \subseteq D$ are left-ideals of the ring $R$,
then the reduction of the canonical $n$-frame
of $L(_RR^n)$ by $b=Be_1 \leq  d=De_1$
 is given by
\[a'_\bot=\sum_{i=1}^n Be_i,\; a'_i=a'_\bot+De_i,\;
c'_{1j}=a'_\bot +D(e_1-e_j),\; a'_\top= \sum_{i=1}^n De_i. \]
See    \cite[Lemma 1.1]{fr2} and Fig.~\ref{figred} for the following.

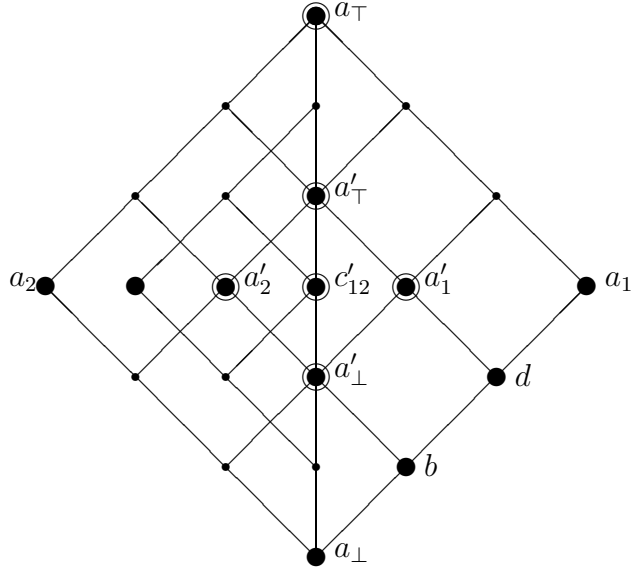
\begin{figure}
\setlength{\unitlength}{12mm}
\begin{picture}(0,7)(0,0)  
\put(0,0){\circle*{0.2}}
\put(0,6){\circle*{0.2}}
\put(-3,3){\circle*{0.2}}
\put(-2,3){\circle*{0.2}}


\put(3,3){\circle*{0.2}}
\put(0,2){\circle*{0.2}}
\put(0,6){\circle*{0.2}}
\put(1,3){\circle*{0.2}}

\put(-1,3){\circle*{0.2}}
\put(0,4){\circle*{0.2}}
\put(0,3){\circle*{0.2}}
\put(0,2){\circle{0.3}}
\put(0,6){\circle{0.3}}
\put(1,3){\circle{0.3}}
\put(-1,3){\circle{0.3}}
\put(0,4){\circle{0.3}}
\put(0,3){\circle{0.3}}
\put(-1,1){\circle*{0.1}}
\put(-2,2){\circle*{0.1}}
\put(1,1){\circle*{0.2}}
\put(2,2){\circle*{0.2}}
\put(0,1){\circle*{0.1}}
\put(0,5){\circle*{0.1}}
\put(-2,4){\circle*{0.1}}
\put(-1,5){\circle*{0.1}}
\put(2,4){\circle*{0.1}}
\put(1,5){\circle*{0.1}}
\put(-1,4){\circle*{0.1}}
\put(-1,2){\circle*{0.1}}
\put(0,1){\line(-1,1){2}}
\put(0,3){\line(-1,1){1}}
\put(0,5){\line(-1,-1){2}}
\put(0,3){\line(-1,-1){1}}
\put(0,0){\line(1,1){3}}
\put(0,0){\line(-1,1){3}}
\put(-1,1){\line(1,1){3}}
\put(1,1){\line(-1,1){3}}
\put(-2,2){\line(1,1){3}}
\put(2,2){\line(-1,1){3}}
\put(-3,3){\line(1,1){3}}
\put(3,3){\line(-1,1){3}}
\put(0,0){\line(0,1){6}}
\put(0.2,6){$a_\top$}
\put(0.2,4){$a'_\top$}
\put(0.2,3){$c'_{12}$}
\put(1.2,3){$a'_1$}
\put(-0.8,3){$a'_2$}
\put(3.2,3){$a_1$}
\put(-2.6,3){$c_{12}$}

\put(-3.5,3){$a_2$}
\put(.2,0){$a_\bot$}
\put(.2,2){$a'_\bot$}
\put(1.2,.9){$b$}
\put(2.2,1.9){$d$}
\end{picture}

\caption{Reduction $2$-frame $(\Phi,a_i,c_{12})$ to $(\Phi^b_d, a'_i,c'_{12})$}\label{figred}
\end{figure}.

\begin{lem}\lab{redframe}
For any  $n$-frame $\Phi$ and $a_\bot\leq  b\leq d \leq a_1$  in a modular lattice, $L$, one has the following
\begin{enumerate}
\item  $\Phi_d^b$ is  an $n$-frame in $L$. 
\item
  $\di ,b \nearrow  (a_1)^{\Phi^b_d},
(a_\bot)^{\Phi^b_d} \nearrow \sum \Phi_\di^b, \sum_{i\neq 1}a_i^{\Phi_\di^b}  
\nearrow$\\ $\nearrow \di + \sum_{i\neq 1}a_i , \bi+ \sum_{i\neq 1}a_i$.
 \item $  a_\top^{\Phi^b_d} \cdot
 \sum \Phi_{\neq 1},\;  a_\bot^{\Phi^b_d } \cdot \sum \Phi_{\neq 1}   
\nearrow (\Phi^b_d)_{\neq 1} ,a_\bot^{\Phi^b_d}$.
\item $(\Phi_{\neq k})_\di^\bi \nearrow (\Phi_\di^\bi)_{\neq k}
\nearrow (a_k)^{\Phi^b_d}+(\Phi_\di^\bi)_{\neq k} \nearrow (a_k+\Phi_{\neq k})_{a_k+\di}^{a_k+\bi}  $\\ for $k>1$.. 
\item If $b=a_\bot$ and $d=a_1$ then $\Phi^b_d=\Phi$.
\end{enumerate}
\end{lem}

\subsection{Towers of $2$-frames}

 An $n$-\emph{tower of $2$-frames}
is a presentation $\Delta(n)=\Delta(2,n)$ 
which is the disjoint  union of $2$-frames $(\Phi^k,a_\bot^k,a^k_i,c^k_{1j})$, $k=1, \ldots ,n,$ 
with the additional relations
\[    a_\top^k, a_2^k \nearrow a_1^{k+1}, a_\bot^{k+1} \mbox{ for } 1\leq k <n. \]
It follows that $a^n_2 a^1_1 =a^1_\bot$ and $a^n_2+a^1_1= a_\top^n$.
Referring to the  reduction setups $\Phi^k(x,y)$
of the $2$-frames $\Phi^k$, define the \emph{reduction setup}
 $\Delta(n)(x,y)$ as the union of the $\Phi^k(x+a_\bot^k,y+a_\bot^k)$
and  the \emph{reduction} $\Delta(n)^b_d=\Delta(n)(b,d)$.
 The following is due to Alan Day  \cite[Thm.5.1]{alan}
\begin{lem}\lab{2tower}
\begin{itemize}
\item[(i)]  Within $\mc{M}$, 
$n$-towers  of $2$-frames are projective.
\item[(ii)]
 $F\mc{M}(\Delta(n))$
is the disjoint union of $5$-element interval sublattices
$\Phi^k\cup  \{a^k_\top\}$.
\item[(iii)]
 $F\mc{M}(\Delta(n))$ is  generated by the $n+2$-elements
$a_1^1, a_2^n, c_{12}^k (1\leq k \leq n)$.
\item[(iv)]
If $\Delta(n)$
is an $n$-tower of $2$-frames in a modular lattice $L$
and if $a^1_\bot \leq b  \leq d\leq a^1_1$
then $\Delta(n)_d^b $
is an $n$-tower of $2$-frames
$a'^k_\bot,a'^k_1,a'^k_2,c'^k_{12}$ in $L$
and with $u=a^n_2\geq u''=ua'^n_2\geq u'= ua'^1_\bot$
and $w=a^1_1\geq w''=d \geq w'=b \geq a^1_\bot$ 
one has $uw=a^1_\bot$, $a'^1_\bot =u'+w'$, and
$a'^n_\top=u''+w''$. Moreover, if $d=a^1_1$ then
$w''=w$ and $u''=u$. If, in addition, $b=a^1_\bot$ then
$\Delta(n)^b_d=\Delta$.
\end{itemize}
\end{lem}
\begin{proof} (i)-(iii) are in \cite{alan}.
(iv) follows from (1) and (2) of Lemma~\ref{redframe}, readily.
See Fig.~\ref{alanfig}.
\end{proof}

 \begin{figure}
\setlength{\unitlength}{8mm}
\begin{picture}(7,10)(-1,-3.6) 
\put(0,0){\circle*{0.2}}
\put(0,6){\circle*{0.2}}
\put(-3,3){\circle*{0.2}}
\put(-2,3){\circle*{0.2}}
\put(3,3){\circle*{0.2}}
\put(0,2){\circle*{0.2}}
\put(0,6){\circle*{0.2}}
\put(1,3){\circle*{0.2}}
\put(-1,3){\circle*{0.2}}
\put(0,4){\circle*{0.2}}
\put(0,3){\circle*{0.2}}
\put(0,2){\circle{0.3}}
\put(0,6){\circle{0.3}}
\put(1,3){\circle{0.3}}
\put(-1,3){\circle{0.3}}
\put(0,4){\circle{0.3}}
\put(0,3){\circle{0.3}}
\put(-1,1){\circle*{0.1}}
\put(-2,2){\circle*{0.1}}
\put(1,1){\circle*{0.2}}
\put(2,2){\circle*{0.2}}
\put(0,1){\circle*{0.1}}
\put(0,5){\circle*{0.1}}
\put(-2,4){\circle*{0.1}}
\put(-1,5){\circle*{0.1}}
\put(2,4){\circle*{0.1}}
\put(1,5){\circle*{0.1}}
\put(-1,4){\circle*{0.1}}
\put(-1,2){\circle*{0.1}}
\put(0,1){\line(-1,1){2}}
\put(0,3){\line(-1,1){1}}
\put(0,5){\line(-1,-1){2}}
\put(0,3){\line(-1,-1){1}}
\put(0,0){\line(1,1){3}}
\put(0,0){\line(-1,1){3}}
\put(-1,1){\line(1,1){3}}
\put(1,1){\line(-1,1){3}}
\put(-2,2){\line(1,1){3}}
\put(2,2){\line(-1,1){3}}
\put(-3,3){\line(1,1){3}}
\put(3,3){\line(-1,1){3}}
\put(0,0){\line(0,1){6}}
\put(0.2,6){$a^2_\top$}
\put(0.2,4){$a'^2_\top$}
\put(0.2,3){$c'^2_{12}$}
\put(1.2,3){$a'^2_1$}
\put(-0.8,3){$a'^2_2$}
\put(3.2,3){$a^2_1$}
\put(-3.6,3){$a^2_2$}
\put(.2,0){$a^2_\bot$}
\put(.2,2){$a'^2_\bot$}
\put(1.2,.9){$b+a^2_\bot$}
\put(2.2,1.9){$d+a^2_\bot$}

\put(4,-4){\circle*{0.2}}
\put(4,2){\circle*{0.2}}
\put(1,-1){\circle*{0.2}}
\put(2,-1){\circle*{0.2}}
\put(7,-1){\circle*{0.2}}
\put(4,-2){\circle*{0.2}}
\put(4,2){\circle*{0.2}}
\put(5,-1){\circle*{0.2}}
\put(3,-1){\circle*{0.2}}
\put(4,0){\circle*{0.2}}
\put(4,-1){\circle*{0.2}}
\put(4,-2){\circle{0.3}}
\put(4,2){\circle{0.3}}
\put(5,-1){\circle{0.3}}
\put(3,-1){\circle{0.3}}
\put(4,0){\circle{0.3}}
\put(4,-1){\circle{0.3}}
\put(3,-1){\circle*{0.1}}
\put(2,-2){\circle*{0.1}}

\put(5,-3){\circle*{0.2}}
\put(6,-2){\circle*{0.2}}
\put(4,-3){\circle*{0.1}}
\put(4,1){\circle*{0.1}}
\put(2,0){\circle*{0.1}}
\put(3,1){\circle*{0.1}}
\put(6,0){\circle*{0.1}}
\put(5,1){\circle*{0.1}}
\put(3,0){\circle*{0.1}}
\put(3,-2){\circle*{0.1}}

\put(4,-3){\line(-1,1){2}}
\put(4,-1){\line(-1,1){1}}
\put(4,1){\line(-1,-1){2}}
\put(4,-1){\line(-1,-1){1}}
\put(4,-4){\line(1,1){3}}
\put(4,-4){\line(-1,1){3}}
\put(3,-3){\line(1,1){3}}
\put(5,-3){\line(-1,1){3}}
\put(2,-2){\line(1,1){3}}
\put(6,-2){\line(-1,1){3}}
\put(1,-1){\line(1,1){3}}
\put(7,-1){\line(-1,1){3}}
\put(4,-4){\line(0,1){6}}

\put(4.2,2){$a^1_\top$}
\put(4.2,0){$a'^1_\top$}
\put(4.2,-1){$c'^1_{12}$}

\put(5.2,-1){$a'^1_1$}
\put(3.2,-1){$a'^1_2$}
\put(7.2,-1){$a^1_1$}
\put(.4,-1.2){$a^1_2$}
\put(4.2,-2){$a'^1_\bot$}
\put(5.2,-3.1){$b$}
\put(6.2,-2.1){$d$}

\put(4.2,-4){$a^1_\bot$}

\put(0,0){\line(1,-1){1}}

\put(1,1){\line(1,-1){1}}

\put(2,2){\line(1,-1){1}}

\put(3,3){\line(1,-1){1}}

\end{picture}

\caption{Reduction of $\Delta(2)$}\label{alanfig}
\end{figure}
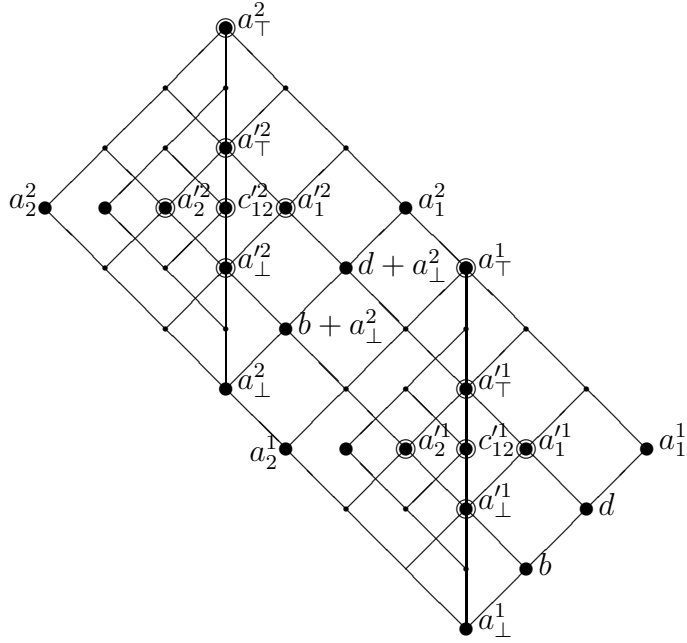.

\subsection{Towers of $3$-frames}
An $n$-\emph{tower} of $3$-frames is a presentation
$\Delta(3,n)$ consisting of the product of an $n$-tower
$\Delta(n)$ of $2$-frames with the chain $a_\bot^1\leq  a_3^1$
and an additional generator $c_{13}^1$ such
that the $a_\bot^1, a_1^1, a_3^1, c_{13}^1 $ form a $2$-frame $\Phi$.
In particular,
by Fact~\ref{mod2}  $a_2^n(a^1_1+a^1_3)= a^1_\bot=(a^1_1+a_2^1)a^1_3$ 
whence $\Phi^1$ together with  $a^1_3, c^1_{13}$ forms a $3$-frame.
See Fig.~\ref{3towfig}

\begin{lem} \lab{3tower}
$n$-towers of $3$-frames are projective within $\mc{M}$.
\end{lem}

\begin{proof}
By Facts~\ref{prod},  \ref{newgen}, and Lemma~\ref{2tower},
the product of $\Delta(n)$ with the chain $a_\bot^1 \leq a_3^1$  
is projective within $\mc{M}$ and strengthening 
with 
\[ c_{13}^1:=  (c_{13}^1 +a_\bot^1)(a_1^1+ a_3^1)  \] 
yields the additional relations $a_\bot^1 \leq c_{13}^1 \leq a_1^1+a_3^1$. 
Now, in view of Fact~\ref{mod3}   put
\[ b= a_1^1c_{13}^1 \mbox{ and }  d= a_1^1(c_{13}^1+ a_3^1)\]
\[(*)\quad v=a^1_3\geq v''=a^1_3(a^1_1+c^1_{13}) \geq v'= a^1_3c^1_{13} \geq a^1_\bot   \]
to obtain the $2$-frame $\Phi'=(b+v',d+b+v', v''+ b, c^1_{13}+b+v', d+v'')$.
Now, together with the chains defined in (iv) of Lemma~\ref{2tower},
 apply Fact~\ref{mod2} to obtain the $n$-tower
$v'+\Delta(n)^b_d$ spanning $[u'+w'+v', u''+w'']$
and the $2$-frame $u'+\Phi'$ spanning $[u'+w'+v', w''+v'']$.
This verifies the strengthening
\[ \Delta(n):=v'+ \Delta(n)_d^b,\, a^1_3:=u'+w'+v'+a^1_3,\, c^1_{13}:=u'+w'+v'+c^1_{13}   \]
and proves the lemma.
\end{proof}

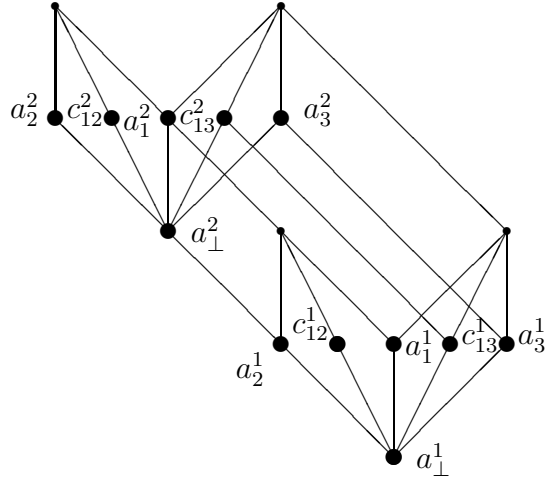
\begin{figure} 
\setlength{\unitlength}{15mm}
\begin{picture}(6,4)(-2,-4)  
\put(0,-2){\line(-1,1){1}} 
\put(2,-3){\line(-1,1){3}}
\put(1,-3){\line(0,1){1}}
\put(0,-2){\line(0,1){1}}
\put(-1,-1){\line(0,1){1}}
\put(2,-3){\line(0,-1){1}}
\put(0,-2){\line(1,-1){2}}
\put(2,-4){\line(1,1){1}}
\put(2,-3){\line(1,1){1}}
\put(2,-4){\line(1,2){1}}
\put(2,-4){\line(-1,2){1}}
\put(0,-2){\line(-1,2){1}}
\put(3,-3){\line(0,1){1}}
\put(2,-4){\circle*{0.15}}
\put(2,-4){\circle*{0.15}}
\put(1,-3){\circle*{0.15}}
\put(2,-3){\circle*{0.15}}
\put(0,-2){\circle*{0.15}}
\put(1.5,-3){\circle*{0.15}}
\put(3,-3){\circle*{0.15}}
\put(2.5,-3){\circle*{0.15}}
\put(1.5,-3){\circle*{0.15}}
\put(0,-1){\circle*{0.15}}
\put(-1,-1){\circle*{0.15}}
\put(-1,0){\circle*{0.07}}
\put(1,-2){\circle*{0.07}}
\put(3,-2){\circle*{0.07}}
\put(-.5,-1){\circle*{0.15}}
\put(0.2,-2.1){$a^2_\bot$}
\put(-0.4,-1.1){$a^2_1$}
\put(-1.4,-1){$a^2_2$}
\put(-0.9,-1){$c^2_{12}$}
\put(2.1,-3.1){$a^1_1$}
\put(3.1,-3){$a^1_3$}
\put(.6,-3.3){$a^1_2$}
\put(1.1,-2.9){$c^1_{12}$}
\put(2.2,-4.1){$a^1_\bot$}
\put(0,-2){\line(1,1){1}}
\put(0,-1){\line(1,1){1}}
\put(0,-2){\line(1,2){1}}
\put(1,-1){\line(0,1){1}}
\put(1,-1){\line(1,-1){2}}
\put(.5,-1){\line(1,-1){2}}
\put(1,0){\line(1,-1){2}}

\put(1,-1){\circle*{0.15}}
\put(.5,-1){\circle*{0.15}}
\put(1,0){\circle*{0.07}}
\put(1.2,-1){$a^2_3$}
\put(.1,-1.05){$c^2_{13}$}
\put(2.6,-3){$c^1_{13}$}
\end{picture}
\caption{$2$-tower of $3$-frames }\label{3towfig}
\end{figure}

The  \emph{reduction setup}  $\Delta(3,n)(x,y)$
is the union of $\Delta(n)(x,y)$ and $\Phi(x,y)$.
\begin{cor}\lab{red3}
\begin{itemize}
\item[(i)] If $\Delta(3,n)$ is an $n$-tower of $3$-frames in a modular lattice $L$
and if $b,d \in L$ such that  $a^1_\bot\leq b \leq d \leq a^1_1$
then $\Delta(3,n)^b_d:= \Delta(3,n)(b,d)$ is also an $n$-tower of $3$-frames in $L$
and spans the interval $[u'+w'+v', u''+w''+v'']$
with the chains from $(*)$ and (iv) of Lemma~\ref{2tower}.
\item[(ii)]
Redefining two of the chains into one,
namely $u:=u+v \geq u'':=u''+v'' \geq u':= u''+v''$, so
one has that $\Delta(3,n)^b_d$ spans the interval 
$[u'+w', u''+w'']$. Moreover, if $d=a_1^1$ then $w=w''$ and $u=u''$.
\end{itemize}
\end{cor}

\subsection{Towers of skew frames}\lab{tow4}
A  \emph{skew} $(n+1,n)$-\emph{frame}
is the  lattice presentation $\Psi$  given by the product
of an $n$-frame $\Phi$ with the chain $a_\bot \leq a'_{n+1}$
and an additional generator $c'_{1,n+1}$ 
such that the list of terms  $a_\bot, b=a_1(a'_{n+1}+c'_{1,n+1}), a'_{n+1}, c'_{1,n+1}$
satisfies the relations of a $2$-frame. See Fig.~\ref{skewframefig}.

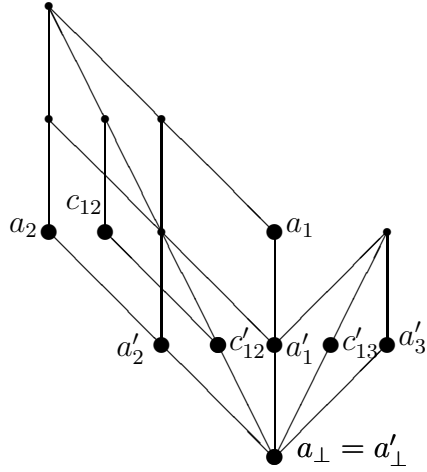
\begin{figure} 
\setlength{\unitlength}{15mm}
\begin{picture}(5,5)(-1,-4)  

\put(0,0){\line(0,-1){2}}
\put(2,-2){\line(0,-1){2}}
\put(1,-1){\line(0,-1){2}}
\put(0,-1){\line(1,-1){2}}
\put(0,0){\line(1,-1){2}}
\put(0,-2){\line(1,-1){2}}
\put(2,-4){\line(1,1){1}}
\put(2,-3){\line(1,1){1}}

\put(2,-4){\line(1,2){1}}

\put(2,-4){\line(-1,2){2}}
\put(1.5,-3){\line(-1,1){1}}
\put(0.5,-2){\line(0,1){1}}
\put(3,-3){\line(0,1){1}}

\put(2,-4){\circle*{0.15}}

\put(2,-4){\circle*{0.15}}
\put(1,-3){\circle*{0.15}}
\put(0,-2){\circle*{0.15}}
\put(0,0){\circle*{0.07}}
\put(0.5,-2){\circle*{0.15}}
\put(2,-2){\circle*{0.15}}
\put(1.5,-3){\circle*{0.15}}

\put(3,-3){\circle*{0.15}}
\put(2.5,-3){\circle*{0.15}}
\put(1.5,-3){\circle*{0.15}}
\put(2,-3){\circle*{0.15}}

\put(0,-1){\circle*{0.07}}

\put(0.5,-1){\circle*{0.07}}

\put(1,-1){\circle*{0.07}}

\put(1,-2){\circle*{0.07}}

\put(3,-2){\circle*{0.07}}

\put(-.35,-2){$a_2$}
\put(.15,-1.8){$c_{12}$}
\put(2.1,-2){$a_1$}
\put(2.1,-3.1){$a'_1$}
\put(3.1,-3){$a'_3$}

\put(.6,-3.1){$a'_2$}

\put(1.6,-3.05){$c'_{12}$}

\put(2.6,-3.05){$c'_{13}$}

\put(2.2,-4){$a_\bot=a'_\bot$}

\put(2.2,-4){$a_\bot=a'_\bot$}

\end{picture}

\caption{Skew $(3,2)$-frame $(\Phi',a'_i,c'_{1j}; \Phi, a_i,c_{12}) $}\label{skewframefig}
\end{figure}

An $n$-\emph{tower} $\Omega(m,n)$ \emph{of  skew}   
 $(m,m-1)$-\emph{frames}   is the presentation given as
the product of an $n$-tower $\Delta(m-1,n)$ of $m-1$-frames 
with  the chain $a^1_\bot \leq  a_m'$ and an additional
generator $c_{1m}'$ subject to relations
stating that $a^1_\bot,a_1^1(a'_m+c_{1m}'),a'_m, c'_{1m}$ is a $2$-frame.
Compare  Fig.~\ref{skewtowerfig} for the case of towers
of skew $(3,2)$-frames. Dealing with the case $m=4$, 
we put $\Omega(n)=\Omega(4,n)$ and speak of $n$-\emph{towers
of skew frames}.

\begin{thm}\lab{4tower}
$n$-towers of  skew frames
can be  defined in terms of $n+6$ generators and
 are projective within $\mc{M}$.
\end{thm}
\begin{proof}
$n+2$ generators  provide the $n$-tower of $2$-frames, $2$ more 
the $n$-tower of $3$-frames, and another $2$ are used to
turn this into an $n$-tower of skew frames.

Omitting the relations concerning $c'_{14}$,
projectivity within $\mc{M}$ follows from  Lemma~\ref{3tower} and Facts~\ref{prod}, \ref{newgen}.  A first strengthening 
\[c'_{14}:= (c'_{14}+a_\bot^1)( a_1^1+a'_4)  \]
adds  the relations 
$ a_\bot^1 \leq c'_{14} \leq  a_1^1+a'_4$. 
In view of Fact~\ref{mod3}  put $b= a_1^1 c'_{14}$ and  
$v=v''=a'^1_4(a^1_1+c'^1_{14}) \geq v'=a'^1_4 c'^1_4$
and let $M$ denote the interval $[b+v', a_1^1+v']$
of the sublattice generated by $a^1_1, a'^1_4, c'^1_{14}$.
Consider the reduction $\Delta(3,n)_{a^1_1}^b$
and 
apply Fact~\ref{mod2}
to the chains $u'=u''\geq u'$ and $w=w''\geq w'$ 
from (ii) of Cor.\ref{red3} and 
$v=v''\geq v'$.  This yields the $n$-tower $v'+\Delta(3,n)^b_{a^1_1}$
spanning $[u'+v'+w',u+v'+w]$    and $u'+M$ 
spanning $[u'+v'+w', u'+v+w]$ and verifies the 
final strengthening
\[  \Delta(3,n):=\Delta(3,n)^b_{a^1_1},\;
c'^1_{14}:= c'^1_{14}+u',\;  a'^1_4:=(a^1_1+u'+v')(c'^1_{14}+v+u').\]    
\end{proof}

The  presentations  
 $n$-tower of $2$- resp.  $3$-frames and
$n$-tower of skew $(4,3)$-frames have been constructed
explicitly and  uniformly for
all $n$.
This results in the following.   
\begin{cor}\lab{4pro}
There is an algorithm constructing for each $n$ 
the presentation
 $\Omega(n)$,   projective
within the class of modular lattices.
\end{cor}

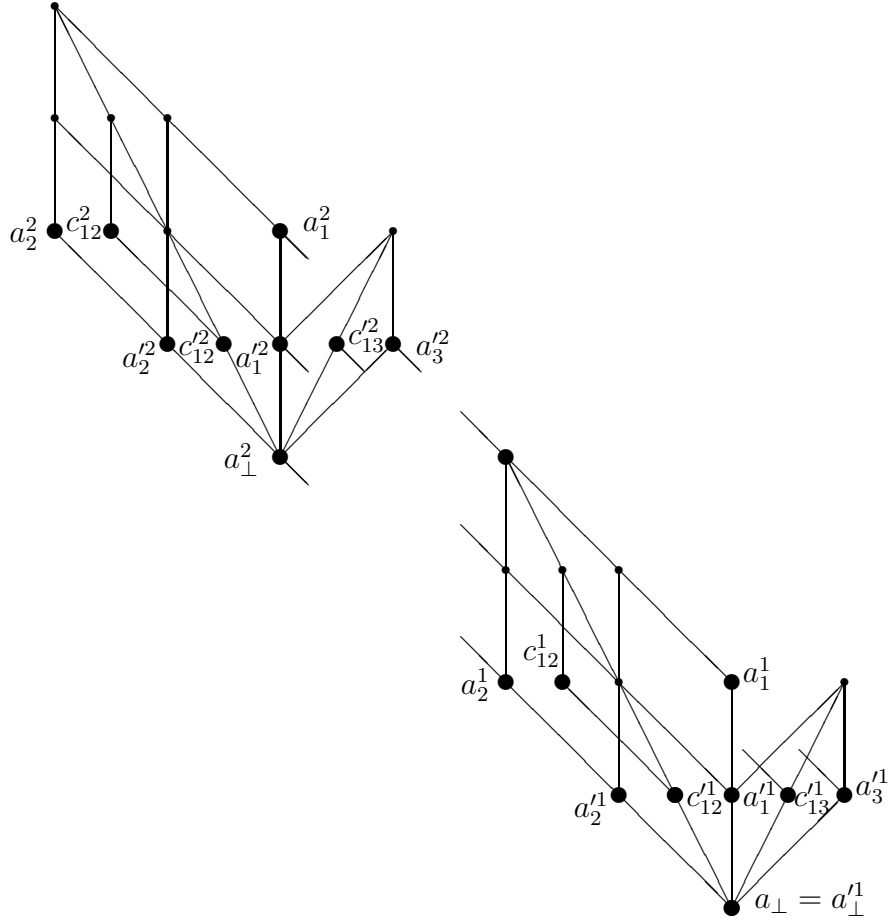
\begin{figure} 
\setlength{\unitlength}{15mm}
\begin{picture}(8,8)(-4,-4)

\put(-4,2){\circle*{0.15}}

\put(-3.5,2){\circle*{0.15}}

\put(-2.5,1){\circle*{0.15}}
\put(-2.,1){\circle*{0.15}}

\put(-1.5,1){\circle*{0.15}}
\put(-1,1){\circle*{0.15}}

\put(-1,2){\circle*{0.07}}

\put(-4,3){\circle*{0.07}}

\put(-4,4){\circle*{0.07}}

\put(-3.5,3){\circle*{0.07}}

\put(-2,0){\line(-1,1){2}}
\put(-2,1){\line(-1,1){2}}
\put(-2,2){\line(-1,1){2}}
\put(-4,2){\line(0,1){2}}

\put(-2,0){\line(1,-1){.25}}
\put(-2,1){\line(1,-1){.25}}
\put(-2,2){\line(1,-1){.25}}
\put(-1.5,1){\line(1,-1){.25}}
\put(-1,1){\line(1,-1){.25}}

\put(-2,1){\line(1,1){1}}

\put(-2,0){\line(1,2){1}}

\put(-2,0){\line(-1,2){2}}

\put(-2.5,1){\line(-1,1){1}}

\put(-3.5,2){\line(0,1){1}}

\put(-1,1){\line(0,1){1}}

\put(0,0){\line(-1,1){.4}}

\put(-2,2){\line(0,-1){2}}


\put(0,-2){\line(-1,1){.4}}
\put(0,-1){\line(-1,1){.4}}

\put(-3,3){\line(0,-1){2}}




\put(-3,3){\circle*{0.07}}

\put(-3,2){\circle*{0.07}}

\put(-3,1){\circle*{0.15}}



\put(-2,2){\circle*{0.07}}

\put(-2,1){\circle*{0.07}}

\put(-2,0){\circle*{0.15}}




\put(0,0){\line(0,-1){2}}
\put(2,-2){\line(0,-1){2}}
\put(1,-1){\line(0,-1){2}}
\put(0,-1){\line(1,-1){2}}
\put(0,0){\line(1,-1){2}}
\put(0,-2){\line(1,-1){2}}
\put(2,-4){\line(1,1){1}}
\put(2,-3){\line(1,1){1}}

\put(2,-4){\line(1,2){1}}

\put(2,-4){\line(-1,2){2}}
\put(1.5,-3){\line(-1,1){1}}
\put(0.5,-2){\line(0,1){1}}
\put(3,-3){\line(0,1){1}}

\put(2,-4){\circle*{0.15}}

\put(2,-4){\circle*{0.15}}
\put(1,-3){\circle*{0.15}}
\put(0,-2){\circle*{0.15}}
\put(0,0){\circle*{0.15}}
\put(0.5,-2){\circle*{0.15}}
\put(2,-2){\circle*{0.15}}
\put(1.5,-3){\circle*{0.15}}

\put(3,-3){\circle*{0.15}}
\put(2.5,-3){\circle*{0.15}}
\put(1.5,-3){\circle*{0.15}}
\put(2,-3){\circle*{0.15}}

\put(0,-1){\circle*{0.07}}

\put(0.5,-1){\circle*{0.07}}

\put(1,-1){\circle*{0.07}}

\put(1,-2){\circle*{0.07}}

\put(3,-2){\circle*{0.07}}

\put(-.4,-2.1){$a^1_2$}
\put(.15,-1.8){$c^1_{12}$}
\put(2.1,-2){$a^1_1$}
\put(2.1,-3.1){$a'^1_1$}

\put(3,-3){\line(-1,1){.4}}
\put(2.5,-3){\line(-1,1){.4}}

\put(3.1,-3){$a'^1_3$}

\put(.6,-3.2){$a'^1_2$}

\put(1.6,-3,1){$c'^1_{12}$}

\put(2.55,-3.1){$c'^1_{13}$}

\put(2.2,-4){$a_\bot=a'^1_\bot$}

\put(-2,2){\circle*{0.15}}
\put(-2,0){\line(1,1){1}}

\put(-4.4,1.9){$a^2_2$}
\put(-3.9,2){$c^2_{12}$}
\put(-1.8,2){$a^2_1$}
\put(-3.4,.8){$a'^2_2$}
\put(-.8,.9){$a'^2_3$}

\put(-2.4,0.8){$a'^2_1$}

\put(-2.9,.9){$c'^2_{12}$}

\put(-1.4,1){$c'^2_{13}$}

\put(-2.5,-.1){$a^2_\bot$}

\end{picture}
\caption{$2$-tower of skew $(3,2)$-frames}\label{skewtowerfig}
\end{figure}

\section{Structure of towers of skew frames}\lab{4}

While projectivity within $\mc{M}$  of the presentation 
``$n$-tower (of skew frames)'' has 
been established in Theorem~\ref{4tower},
in this section we collect the needed concepts and results
about models of this and related presentations -
considered as lists  of elements or \emph{configurations}.
Maps are supposed to match such lists.
If such list is the concatenation of parts we say
that it is obtained by  \emph{combining} these parts.
We will use $\bar a$ to denote
lists of elements, in general, not necessarily the
  $a_1, \ldots ,a_n$  of an $n$-frame.

 First, we recall some
more details about particular elements in
sublattices generated by (skew)  frames  - 
used in  equivalent presentations e.g. in \cite{fm4}.

\subsection{Frames}\lab{framme}
A well known equivalent definition of frames 
is obtained es follows cf.  \cite{fr1}.
Consider an $n$-frame $\Phi$ in a modular lattice $L$ and define
$c_{ij}=c_{ji}= (a_i+a_j)(c_{1i}+c_{1j})$
for $1 \neq i \neq j\neq 1 $, Then it follows
\[ (\sum_{i \in I}a_i) \cdot (\sum_{j \in J} a_j) =
\bigl\{ \begin{array}{ll} a_\bot & \mbox{ if  } I\cap J=\emptyset \\ 
\sum_{k \in I\cap J}a_k & \mbox{ else} \end{array}
\; \mbox{ for } I,J \subseteq \{1,\ldots ,n\},   \]
and for pairwise distinct $i,j,k$ 
\[  a_i+a_j =a_i+ c_{ij},\;\;
 a_i \cdot c_{ij} =a_\bot,\]\[
c_{ik}=(a_i+a_k)\cdot (c_{ij}+c_{jk}). \]
We also write $\bot^\Phi=a_\bot^\Phi, a_i^\Phi, c_{ij}^\Phi$ for these elements
and $\top^\Phi=\sum_{i=1}^n a_i$.
Observe that, for any $k$, the $a_i,c_{ij}$ with $i,j \neq k$
form an $n-1$-frame $\Phi_{\neq k}$ in $L$.
An important property of frames $\Phi$ is the existence of the
\emph{perspectivities} $\pi_{kl}=\pi_{kl}^\Phi$, $k\neq l$, that is   lattice isomorphisms
between intervals of $L$ 
matching $\Phi_{\neq k}$ with $ \Phi_{\neq l}$
\[ \pi_{kl}:[a_\bot, \sum_{i \neq k}a_i] 
\to [a_\bot, \sum_{i \neq l}a_i]  
\mbox{ where } \pi_{kl}(x)= (x+c_{kl})\sum_{i \neq l}a_i.    \] 
Thus, $\pi_{kl}^\Phi(a)$ is obtained from a  lattice term 
$\hat{\pi}_{kl}(x,\bar z)$  substituting $a$ for $x$ and $\Phi$ for 
$\bar z$ (actually, we use only the case where $\Phi$ is a $4$-frame.)

\subsection{Reduction of frames}
Given 
a frame $\Phi$ and $b_1\in L$  such that   $a_\bot \leq b_1 \leq a_1$, define
\[b_j=a_j(b_1+c_{1j}) \mbox{ for } j\neq 1, b=\sum_{j=1}^n b_i,
 \mbox{ and }  b_{ij}=(b_i+b_j)c_{ij}   \]
to obtain, by \emph{upper reduction} with $b_1$,
the frame 
\[\Phi^{b_1}=\Phi^{b_1}_{a_\bot}=(b, b+a_{i (1\leq i \leq n)}, 
b+c_{ij( i\neq j)})  \] 
and by  \emph{lower reduction} the frame
\[ \Phi_{b_1} = \Phi^{a_1}_{b_1}= (a_\bot, b_{i (1\leq i\leq n)} ,b_{ij  (i\neq j)}).\] 
In view of the perspectivities, in both cases  the resulting frame is the same
if, for some $i\neq 1$, the construction is carried out based on 
$a_\bot \leq b_i \leq a_i$ 
such that $b_1= a_1(b_i+c_{1i})$.

\subsection{Stable elements}
Given a modular lattice $L$,  element $s$ of $L$ and an 
$n$-frame $\Phi$ in $L$, the element $s$
is $j$-\emph{stable} in $L$ for $\Phi$ and $j\geq 2$ if
\[sa_1=sa_j=a_\bot  \mbox{ and } s+a_1=s+a_j= a_1+a_j \]
and for all 
 $b_1\in L$ with  $a_\bot\leq b_1 \leq a_1$ one has
\[ s+ b_1 = s+b_j  \mbox{ where } b_j=a_j( b_1+c_{1j}).     \]
Obviously,  if $\Phi'$ is
another $n$-frame in $L$ such that
 $a_\bot, a_1,a_j,c_{1j}\nearrow  a'_\bot, a'_1,a'_j,c'_{1j}$   in $L$
then $s$ is $j$-stable for $\Phi$ if and only if 
$s+a'_\bot$ is $j$-stable for $\Phi'$.
Also,
in view of the perspectivities, if  $s$ is $j$-stable for $\Phi$ then
$\pi_{jk}(s)$ is $k$-stable for $\Phi$. 
This crucial concept is due to Ralph Freese \cite{fr}.
See \cite[Lemma 2,3]{fm4} for the following
\begin{fact}\lab{stab}
If $s$ is $j$-stable in $L$ for $\Phi$  then
 for all
$b \in L$ with  $a_\bot \leq b \leq a_1$, one has
$s+ a_\bot^{\Phi^b}$ $j$-stable for $\Phi^{b}$ and $s a_\top^{\Phi_b}$ 
$j$-stable for $\Phi_{b}$.
\end{fact}
.

\subsection{Skew frames}\lab{skew}
Using the above view (Subsection~\ref{framme})  on frames in modular lattices,
an equivalent definition of
a skew $(n+1,n)$-frame $\Psi=(\Phi',\Phi)$ in a modular lattice $L$ 
is that of a  configuration which is composed 
by the $n$-frame $\Phi$ and the $n+1$-frame
$\Phi'$ such that $\Phi'_{\neq n}$ is the reduction $\Phi_{a'_1}$,
in particular $a_\bot=a'_\bot$.
Observe that the perspectivity $\pi_{kl}$ 
of $\Phi$ induces the perspectivity $\pi_{kl}$ of $\Phi'_{\neq n}$.
We will deal with the cases $(3,2)$ and $(4,3)$, only. 

Our basic example of a skew $(4,3)$-frame is as follows: For a prime 
$p$ consider the $\mathbb{Z}$-module
$A$ with generators $e_i$, $i=1,2,3,4$ and relations $p^2e_i=0$
for $i\leq 3$, $pe_4=0$.
Then a skew $(4,3)$-frame in the submodule lattice $L(A)$ is obtained as follows:
Let
$a_\bot=0$,  $a_i=\mathbb{Z} e_i$, $c_{ij}=\mathbb{Z}(e_i-e_j)$,
    $a'_i=\mathbb{Z} pe_i$, $c'_{ij}=\mathbb{Z}(pe_i-pe_j)$,
for $i,j\leq  3$,
$a'_4=\mathbb{Z} e_4$, and $c'_{i4}=\mathbb{Z}(pe_i-e_4)$. \\

Dealing with a skew $(4,3)$-frame $\Psi$
we consider  $\Psi_{(3,2)}$ given by generators not involving index $2$ 
and $\Psi^{(3,2)}= a_2+ \Psi_{(3,2)}$. Observe that 
$\Psi_{(3,2)} \nearrow \Psi^{(3,2)}$ and that 
the relations of a skew $(3,2)$-frame are implied by
those of a skew $(4,3)$-frame, in both cases.

For $a_\bot \leq b \leq a'_1  \leq d \leq a_1$ in $L$
the \emph{lower reduction} $\Psi_{b,d}$ is the skew $(n+1,n)$-frame
which combines  the lower reductions $\Phi'_{b}$ and $\Phi_{d}$.
For $a_\bot \leq b \leq a'_1$ in $L$
the \emph{upper reduction} $\Psi^{b}$ 
combines the upper reductions $\Phi'^{b}$ and $\Phi^{b}$.

\subsection{Towers of skew frames}
\begin{obs}\lab{obs4}
Within $\mc{M}$ a presentation 
equivalent to that of 
an $n$-tower $\Omega(n)$
is given by  a list $\bar a$
combining
 the skew $(4,3)$-frames $\Psi^k=(\Phi'^k,\Phi^k)$ $(1\leq k \leq n)$
each consisting of 
the $4$-frame $\Phi'^k$ and the $3$-frame $\Phi^k$ where
 \[\Phi'^k=\bar a'^k=(a'^k_i,c'^k_{ij}\mid i,j\leq 4, i\neq j),\;
\Phi^k=\bar a^k=(a^k_i,c^k_{ij}\mid i,j\leq 3, i\neq j)\]
such that
\[(*)\; \Psi^k_{(3,2)} \nearrow (\Psi^k)^{(3,2)}
\nearrow \Psi^l_{(3,2)} \nearrow (\Psi^l)^{(3,2)}
\mbox{ for } 1\leq k <l \leq n.\]
\end{obs}
Observe that $\Psi^k_{(3,2)}$ consists of the $a^k_i,c^k_{ij},a'^k_i,c'^k_{ij}$
 where $i,j \neq 2$ and that 
$(\Psi^k)^{(3,2)} = a^k_2 + \Psi^k_{(3,2)} $. 
The exceptional role of index $2$ (linked to the ``upward direction'')
comes out of the construction of $n$-towers 
 and fits
to the application of \cite{fm4}.

\begin{proof}
Recall the definition of
 $n$-towers $\Omega(n)$ of skew frames in Subsection~\ref{tow4}
and observe that the $a^1_i, c^1_{ij}, a'^1_i, c'^1_{ij}$ 
form a skew $(4,3)$ frame $\Psi^1$,  that $a_2^n, a_1^1, a_3^1, a'^1_4$ 
are relatively independent over $a^1_\bot$, and that $a^k_1=a^k_\bot +a^1_1$.
Thus, $\Psi^1_{(3,2)}$ is a skew $(3,2)$-frame whence   so is
(for each $k>1$)
$a^k_\bot +\Psi^1_{(3,2)}$ which combines with $a^k_2, c^k_{12}$ 
to form a skew $(4,3)$-frame $\Psi^k$ such that $\Psi^1_{(3,2)} 
\nearrow  \Psi^k_{(3,2)}$. This
 proves $(*)$ in view of the remarks following Fact~\ref{mod}
and applies also to skew frame as considered in the preceding
Subsection.
\end{proof}

A model is obtained from the submodule lattice of the
free $\mathbb{Z}$ 
module with generators $e_1,e_2,e_3,e_4$ and relations
$p^2e_1=0$,
$ p^{3n-1} e_2=0$, $p^2 e_3=0$ and $pe_4=0$. Indeed, put
(where  $i=1,3$)
\[a_\bot^k= \mathbb{Z}p^{3(n-k)+2}e_2,\:
a_2^k = \mathbb{Z}p^{3(n-k)}e_2, \;c_{2i}^k = \mathbb{Z}(p^{3(n-k)}e_2-e_i)
\] 
\[a_i^k = a_\bot^k+ \mathbb{Z}e_i,\;
c^k_{13} = a_\bot^k+ \mathbb{Z}(e_1-e_3),\; 
c_{i2}^k =a_\bot^k+ \mathbb{Z}(e_i-p^{3(n-k)}e_2)\] 
\[a'^k_2 = \mathbb{Z}p^{3(n-k)+1}e_2, \;c'^k_{2i} = 
\mathbb{Z}(p^{3(n-k)+1}e_2-pe_i), 
\;c'^k_{24} = \mathbb{Z}(p^{3(n-k)+1}e_2-e_4),\] 
\[a'^k_i = a_\bot^k+ \mathbb{Z}pe_i,\;
c'^k_{13} = a_\bot^k+ \mathbb{Z}(pe_1-pe_3),\; 
c'^k_{i4} =a_\bot^k+ \mathbb{Z}(pe_i-e_4)
.\]

\subsection{Reduction of towers}
Given an $n$-tower $\Omega(n)=\bar a$ in a modular lattice $L$,
consider fixed  $m>0$ and $b,d \in L$
such that
\[ a^m_\bot \leq b \leq a'^m_1 \leq d \leq a^m_1.    \]
The \emph{lower reduction } $\Omega(n)_{b,d}$ of $\Omega(n)$ combines
the following reductions of skew-frames $\Psi^k=(\Phi'^k,\Phi^k)$
\[\begin{array}{lll}
 (\Phi'^k_{a^k_1 b},&\Phi^k_{a^k_1 d}) &\mbox{ for } 1\leq k<m\\
(\Phi'^m_{b},& \Phi^m_{d})\\
(\Phi'^k_{b+a^k_\bot},& \Phi^k_{d+a^k_\bot})
&\mbox{ for } m<k \leq n \end{array} \]
Given $b \in L$ 
such that \[a^m_\bot\leq b\leq a'^m_1\] the \emph{upper reduction } $\Omega(n)^{b}$ of $\Omega(n)$ combines
the following reductions of skew-frames
$\Psi^k=(\Phi'^k,\Phi^k)$
\[\begin{array}{lll}
 ((\Phi'^k)^{a^k_1 b},&(\Phi^k)^{a^k_1 b}) &\mbox{ for } 1\leq k<m\\
((\Phi'^m)^{b},& (\Phi^m)^{b})\\
((\Phi'^k)^{b+a^k_\bot},& (\Phi^k)^{b+a^k_\bot})
&\mbox{ for } m<k \leq n \end{array} \]
We speak of the lower resp. upper  reduction of $\Omega(n)$ \emph{induced} by the
reduction of $\Psi^k$.

\begin{obs}\lab{towred}
If the configuration $\Omega(n)$ is an $n$-tower of skew frames
$\Psi^k=(a'^k,c'^k_{ij},a^k,c^k_{ij})$ in 
a modular lattice, $L$, then so are any of its
lower reductions $\Omega(n)_{b,d}$ 
where $b,d \in L$   and any of its upper reductions
$\Omega(n)^{b}$ where $b \in L$. 
Moreover, if $\phi:L \to L'$ 
is a homomorphism into a modular lattice $L'$ such that
 $\phi(b)=\phi(a'^m_\bot)$ and $\phi(d)=\phi(a^m_1)$
resp. $\phi(b)= \phi(a^m_\bot)$
then $\phi(\Omega(n))=\phi(\Omega(n)_{b,d})$ resp.
$\phi(\Omega(n))=\phi(\Omega(n)^{b})$, as configurations in $L'$.  
\end{obs}

\section{Coordinates and characteristic}\lab{5}

\subsection{Coordinate ring}\lab{coor}
Following von Neumann \cite{neu} (cf. Freese \cite{fr1,fr2,fr}
and \cite[Lemma 6]{cons}) with
any $4$-frame $\Phi$
in a modular lattice $L$ and choice of $3$ different indices  (here we use
 $1,3,4$)
one obtains a (\emph{coordinate}) ring   $R(\Phi,L)$ with unit $c_{13}$
and zero $a_1$, the elements of which are the $r \in L$
such that $ra_3=a_\bot$ and $r+a_3=a_1+a_3$.
More precisely, there are 
binary lattice terms $x \oplus_{\bar z} y$ and $x \otimes_{\bar z} y$
and a unary term $\ominus_{\bar z} x$ 
defining these coordinate rings. Here, one has $\bar z=(z_i,z_{ij}|
i,j\neq 2)$ corresponding to the $3$-frame $(a_i,c_{ij}| i,j\neq 2)$.
For   given $L$ and $\Phi$ these rings are isomorphic
for any choice
of the triple of indices -  via the perspectivities
resp.  compositions thereof.

If $L$ embeds into the the subgroup lattice of an abelian group $A$
and if $\Phi=(a_i,c_{ij}\mid i,j \neq 2)$ is  a $3$-frame in $L$ then the above 
definitions apply to obtain the  ring $R(\Phi,L)$,
embedded into the endomorphism ring  of the associated subquotient of $A$.

An element $r$ of $R(\Phi,L)$ is invertible if and only if $ra_1=a_\bot$
and $r+a_1=a_1+a_3$; these form the group $R^*(\Phi,L)$
of units in the ring $R(\Phi,L)$.
 Moreover, there is a lattice  term
$t(x,\bar z)$ such that $t(r,\bar a)$ 
is the inverse of $r$ if $r$ is invertible.

\subsection{Stable elements}
Obviously, $3$-stable elements
are invertible.\\
 Again, the following
crucial tool  is due to Ralph Freese \cite{fr}.

\begin{lem}\lab{stabgp}
For a modular lattice $L$ containing a $4$-frame $\Phi$ as above one has the following.
 \begin{itemize}
\item[(i)] The  elements of $L$
 which are $3$-stable for $\Phi$ in $L$ 
form a subgroup $R^\#(\Phi,L)$ of the group  $R^*(\Phi,L)$ of units.
\item[(ii)] For each $b\in L$ with $a_\bot\leq b \leq a_1$
the map $r \mapsto r+  \bot^{\Phi^{b}}$
is a homomorphism  $\beta_{b}:R^\#(\Phi,L) \to R^\#(\Phi^{b},L)$.
 \item[(iii)] If $r$ is $3$-stable for $\Phi$ and
 $b=a_1(r +c_{13})$ then
 $\beta_{b}(r) =
c_{13}^{\Phi^{b}}$ is the unit  of  $R^\#(\Phi^{b},L)$.
\end{itemize}
\end{lem}

\begin{proof}
(i) and (ii) are Lemma 1.3-6 of \cite{fr}.
For convenience, we prove (iii). With $b=\bot^{\Phi^{b}}$ 
 one has $r+b= r+b+b= (a_1+r)(c_{13}+r)+b \geq c_{13}+b$ 
and equality follows since by (ii) both   are complements of $a_3+b$
in $[b,b+a_1+a_3]$.
\end{proof}

\subsection{Characteristic}
With the term $x \oplus_{\bar z} y$ of Subsection 5.1,
define recursively, $1\otimes_{\bar z} z_{14}=z_{14}$ and
$(n+1)\otimes_{\bar z} z_{14} = z_{14}\oplus_{\bar z} (n \otimes_{\bar z} z_{14})$.
In the sequel, $p$ will be a fixed prime.
The $4$-frame $\Phi=\bar a$ has \emph{characteristic} 
$p$  if $p\otimes_{\bar a} c_{14}=a_1$. 
Ralph Freese \cite{fr1} has shown that, for any 
frame $\Phi=\bar a$ in a modular lattice  $L$,
the frame $\Phi_{a_1(p\otimes_{\bar a} c_{14})}$ has characteristic
$p$ - and equals $\Phi$ if $\Phi$ has characteristic $p$, already.

Let $(\bar z',\bar z)$ denote a list of variables
to be used for substituting skew $(4,3)$-frames.
In \cite[p. 516]{fm4}, a term $p_{32}(\bar z)$ has been defined
and a skew $(4,3)$-frame $(\Phi',\Phi)=(\bar a',\bar a)$
 has been called of  
\emph{characteristic} $p \times p$ if $\Phi'$ is
of characteristic $p$ and $p_{32}(\bar a)\geq a_3'$ 
and $a_3+ p_{32}(\bar a) =a'_2+ p_{32}(\bar a)
=a'_2+a_3$. The
following is \cite[Lemma 9]{fm4}
(in  the proof given, there,
observe that  $b_3 \geq a'_3$ since $p_{32} \geq a'_3$). 

. 

\begin{lem}\lab{charp} There are terms $b^*(\bar z',\bar z)$
and $d^*(\bar z',\bar z)$ such 
that for any skew $(4.3)$-frame $\Psi=(\Phi',\Phi)$ in a modular lattice $L$
one has 
\[a_\bot \leq 
{\bf b}:=b^*(\bar a', \bar a) \leq a'_1 \leq
{\bf d}:=d^*(\bar a',\bar a) \leq a_1\]
 and
obtains  $\Psi_{{\bf b},{\bf d}}$ of characteristic $p\times p$.
Moreover, if $\Psi$ has characteristic $p\times p$
then ${\bf b}=a_\bot$ and ${\bf d}=a_1$,
that is $\Psi_{{\bf b},{\bf d}}=\Psi$.
\end{lem}
As in Freese's result, one can derive projectivity 
but, in contrast, it appears unlikely that characteristic
$p\times p$ is preserved under reductions. Though, the existence of stable elements is preserved (see Fact~\ref{stab}).
Recall the term $g_1^*(\bar z',\bar z)$ from 
 \cite[Cor.13]{fm4} and, applying the ``perspectivity term'' $\hat{\pi}_{23}$
 define
\[g^+(\bar z',\bar z)= \hat{\pi}_{23}(g_1^*(\bar z',\bar z),\bar z').\]
\begin{lem}\lab{stabpp}
 For any skew $(4,3)$-frame $\Psi=(\Phi',\Phi)=(\bar a',\bar a)$  
of characteristic $p\times p$ in a modular lattice $L$, 
 one has  
$g^+(\Psi)=g^+(\bar a',\bar a)$
an
 element of $L$ which is $2$-stable
for $\Phi'$. \end{lem}
\begin{proof}
According to  \cite[Cor.13]{fm4}
one has $g_1^*(\bar a',\bar a)$ 
a $2$-stable element of $R(\Phi')$ and the claim follows via perspectivity.  
\end{proof}

\section{Glueing constructions}

Extending early work of Dilworth 
and Hall, analysis and construction of lattices $L$ as unions of interval
sublattices have been studied by several authors, see
\cite{kleb,dh} and \cite[Section 3]{fm4}), also 
 \cite{df} for a survey. Here, we need only the special case 
where $L$ and the ``skeleton'' $S$
 are  finite
modular lattices. Though, without additional effort,
one can allow arbitrary bounded lattices $L$ and
skeletons $S$ which are modular of finite height.

\subsection{Glueing construction of Dilworth and Hall}\lab{Dil}
Given  intervals $L_i=[a_i,b_i]$, $i=1,\ldots ,n$,  in a modular lattice, $L$,
such that $a_i\leq a_{i+1}$ and $b_i \leq b_{i+1}$  for $1\leq i<n$,
the union of these intervals is a sublattice of $L$ and
one has isomorphisms  $\alpha_i:[c_i,b_i] \to [a_{i+1}, d_i]$, $i<n$,
with $\alpha_i(x)=x+a_{i+1}$ 
(and inverse $\alpha_i^{-1}(y)= b_iy$)
 where $c_i=b_ia_{i+1}$ and $d_i = b_i+a_{i+1}$.

Conversely,
given 
pairwise disjoint  modular lattices $L_i=[a_i,b_i]$ ($i\leq n$)
and  
isomorphisms $\alpha_i:[c_i,b_i]_{L_i} \to [a_{i+1]},d_{i+1}]_{L_{i+1}}$ ($i<n$)
where $c_i \in L_i$ and $d_{i+1} \in L_{i+1}$
there is a modular lattice $L$, the \emph{Dilworth-Hall glueing}, which is the union
of interval sublattices $L_i$,  related as above.

Also, one obtains a homomorphic image of $L$
in which the intervals $[c_i,b_i]$ and $[a_{i+1},d_i]$ 
are identified via $\alpha_i$.

\subsection{Decomposition of lattices as glued sums}
In the sequel let $S$ a modular lattice of finite height
with bottom $\zero$ and top $\one$. We
write $x \prec y$ if $x$ is a lower cover of $y$ in $S$.

Consider a  lattice $M$,
a join embedding $\sigma:S\to L$, and a meet embedding
$\pi:S\to L$ such $\sigma y \leq \pi x$ for all $x\prec y$ in $S$.-
Then the union $L$ of  interval sublattices
$L_x=[\sigma x, \pi x]$, $x \in S$. of $L$ is a sublattice of $M$.
 $L$ is called an $S$-\emph{glued sum} of the  $L_x$, $(x \in S)$.
$L$ has greatest element $1=\pi(1)$ and smallest element 
$0=\sigma(0)$.

If $L'$ is another $S$ glued sum given by $\sigma',\pi'$
then $L$ is isomorphic to $L'$ 
if and  if there are isomorphisms $\chi_x:L_x \to L'_x$ such that 
$\chi_x$ and $\chi_y$ both induce, for $x \prec y$,  the same isomorphism
of $L_x \cap L_y$ onto $L'_x\cap L'_y$.

Clearly, if $T=[u,v]$ is an interval sublattice of $S$ then
$\bigcup_{x \in T} L_x$ is a $T$-glued sum and an interval
 sublattice of $L$.
This allows to use induction over the height of $S$.

Observe that given $x < y$ in $S$, $a \in L_x$, and $b \in L_y$, 
one has $a\leq b$ if and only if for some/each
chain $x=x_1 \prec x_2 \ldots \prec x_n=y$ in $S$
one has $a_i \in L_{x_i}\cap L_{x_{i+1}}$
\[ a\leq_{x_1}a_1 \leq_{x_2} a_2 \ldots \leq_{x_{n-2}} a_{n-1}\leq_{x_n} a_n=b.    \]

\begin{claim}\lab{simpglu}
 An $S$-glued sum  $L$ is simple of so are the $L_x$ $(x \in S)$.
\end{claim}
\begin{proof}
For  any homomorphism $\phi$ of $L$ onto $L'$
one has $\phi(\sigma(x))=\phi(\pi(x))$ for all $x \in S$
since the $L_x$ are simple. 
If $x\prec y$ in $S$ then  $\sigma y \leq \pi x$ 
whence $\phi(\sigma(y)) \leq \phi(\pi(x)) =\phi(\sigma(x))$
and it follows $\phi(\sigma(x))=\phi(\sigma(y))$ for all $x \leq y$,
in particular $\phi(0) =\phi(\sigma(0)) =\phi(\sigma(1)) =\phi(\pi(1))=
\phi(1)$.

\end{proof}

\subsection{Modularity}
\begin{lem}\lab{modglu}
Any $S$-glued sum $L$ of modular lattices   $L_x$ is modular.
\end{lem}
\begin{proof} Consider $b\leq a$ and $c$ in $L$ such that 
$ac \leq b\leq a \leq b+c$. Let $a \in L_u$, $b \in L_v$,  and $c \in L_w$.
From $\sigma v \leq b \leq a \leq \pi u$ one obtains $\sigma(uv) \leq \sigma v \leq b \leq (\pi u)(\pi v) =\pi(uv)$ which allows to assume $v \leq u$. 
Now $\pi(v+w) \geq b+c \geq a$ so that w.l.o.g. $v+w=\one$. Dually,
one may assume $uw=\zero$ whence $v=u$ by modularity of $S$. 
If $\zero \prec  x <u$ then $w \prec x+w <\one$  and, by induction,
$[\sigma x, \one]$  and $[\zero, \pi(w+x)]$ are modular.
whence  Dilworth-Hall applies. Similarly, if $\zero \prec x <w$.
This leaves the case that $u,w$ are atoms or $\zero$.
If, say, $u=\zero$ then one has the Dilworth-Hall
glueing $[0,\pi u] \cup [\sigma w, 1]$ whence modularity of $L$.
If both $u,w$ are atoms then $[0, \pi u]$ and $[\sigma w, 1]$
 are modular by Dilworth-Hall and then so is $L$.

To prove the second claim, choose $x\prec \one$;
by inductive hypothesis, there is a maximal chain

\end{proof}

\subsection{Calculations in glued sums}
We write $\sigma x=0_x$ and    $a+_x b=a+ b$ for $a,b \in L_x$.
For maximal chains $C$ in intervals $[x,z]$ of $S$ 
and $a \in L_x$ we define, recursively,
$a +_C 0_x =a $ if $C=\{x\}$ and 
\[ a +_{C} 0_z =(a +_D 0_u)+_u 0_z \mbox{ where } D=C\setminus \{z\}
\mbox{ and } u\prec z, u \in C     \]
\[ a+_C c= (a+_C 0_z)+_z c \mbox{ for } c \in L_z.\]
Now, for $a \in L_x$, $b \in L_y$,
$z=x+y$, and maximal chains $C$ in $[x,z]$ and $D$ in $[y,z]$ define 
\[ a+_{C,D} b = (a+_C 0_z)+_z (b+_D0_z).  \] 
The following is obvious, as is its dual.
\begin{claim}\label{join}\begin{itemize}
\item[(i)] $a+_C c= a+c$ for all $x\leq z$ in $S$,  $a\in L_x$, $c \in L_z$
and maximal chains $C$ in $[x,z]$.
\item[(ii)] $a+_{C,D} b= a+b$ for all $x,y$ in $S$,  $a\in L_x,$, $b \in L_y$,
and maximal chains $C$ in $[x,x+y]$, $D$ in $[y,x+y]$.
\end{itemize}
The dual results hold for meets.
\end{claim}

For our main result it will be crucial that certain calculations 
can be carried out in a partial sublattice of $L$.
We introduce some notation which will be useful later.
For $x\prec y$ in $S$ we put
$0_{y,x}= \sigma y$ and $1_{y,x}=\pi x$.

Now, given $P\subseteq S^2$ where $x\prec y$ for all $(x,y)\in P$ 
let $L_{|P}= \bigcup_{(x,y)\in P} (L_x\cup L_y)$
endowed with the partial operations $ a+_Pb=c$ if and only if
$a,b \in L_x$ for some $(x,y) \in P$ and $c=a+_x b$
 or if $a \in L_x$, $b=0_y$, and $c=a+_x 0_{y,x}$
for some $(x,y) \in P$ or, similarly,   interchanging $a$ with $b$. 
Partial meets are defined, dually.    
In view of Claim~\ref{join} any calculation in $L$
can be composed by  calculations in $L_{|P}$ where 
$(x,y)\in P$ for all $x\prec y$.

\subsection{Glueing of  sets}
Again, $S$ is a finite height modular lattice. In particular,
$S$ is a directed graph with edges $(x,y)$ where $x \prec y$.
Thus, a chain $x_1\prec x_2 \ldots \prec x_n$ 
is a (directed) path from $x_1$ to $x_n$.

A \emph{glueing} of a family  $L_x(x \in S)$ of pairwise disjoint sets
is given by injective partial maps $\gamma_{yx}:L_x \to L_y$,
 $x\prec y$,  such that
\[(*)\quad \gamma_{x+y,x} \gamma_{x,xy} = \gamma_{x+y,y} \gamma_{y,xy} 
\mbox{ for } xy\prec x,y \prec x+y.  \]   
We put $\gamma_{x,x}=\gamma_C$ the identity on $L_x$ where $C=\{x\}$.
Observe that these conditions are satisfied if one replaces 
the order by its dual and the $\gamma_{y,x}$ by their inverses. 
This provides the counterparts of concepts and results.

Now, for a chain $C=\{x_1\prec x_2 \ldots \prec x_n   \}$
define
\[ \gamma_C= \gamma_{x_n,x_{n-1}} \circ \ldots \circ \gamma_{x_2,x_1}:L_{x_1} \to L_{x_n}\]
which is again an injective partial map, possibly empty.

\begin{claim}\lab{6.3}
\begin{itemize} 
\item[(i)] If $C, D$ are  maximal chains in $[x,z]$ then $\gamma_D=\gamma_C$.
\item[(ii)] If $C_i$ is a maximal chain in $[x_{11}x_{21},x_{in_i}]$ for $i=1,2$
then there are maximal chains $D_i$ in $[x_{n_i},x_{1n_1}+x_{2n_2}]$
such that $\gamma_{D_1} \gamma_{C_1}= \gamma_{D_2}\gamma_{C_2}$.
\end{itemize}
\end{claim}
\begin{proof}
To prove (i) we proceed
by induction  on the height of $[x,z]$.
Consider  $x \prec u ,v$, $u \in C$, $v\in D$
and the maximal chains $C'=C \setminus \{u\}$ in
$[u,z]$ and $D'=D\setminus\{v\}$ in $[v,z]$.
 If $u=v$
then
 one has 
$\gamma_C= \gamma_{C'} \gamma_{ux} =\gamma_{D'}\gamma_{ux}=\gamma_D $ 
by inductive hypothesis. Now, assume $u\neq v$ and $w=u+v$.
Choose a maximal chain $E$ in $[w,z]$. Again by induction
and by $(*)$, one has
\[\gamma_C = \gamma_{C'} \gamma_{ux} =
\gamma_E \gamma_{wu} \gamma_{ux} = \gamma_E \gamma_{wv} \gamma_{vx}
= \gamma_{D'} \gamma_{vx} = \gamma_D.\]
To prove (ii) 
let $D_j$ the image of $C_i$ under the isomorphism
$x \mapsto x+x_{jn_j}$ of   $[x_{11}x_{21},x_{in_i}]$ 
onto  $[x_{jn_j},x_{1n_1}+x_{2n_2}]$ for $\{i,j\}=\{1,2\}$.
\end{proof}

For $a_i \in L_{x_i}$ $(i=1,2)$ 
define $a_1\sim a_2$ if and only if there are maximal chains
$C_i$ in $[x_i, x_1+x_2]$  such that 
$\gamma_{C_1} a_1 =\gamma_{C_2} a_2$.
In view of the dual of (ii) in Claim~\ref{6.3}
one has $a_1\sim a_2$ if and only if 
there are maximal chains $D_i$ in $[x_1x_2, x_i]$
such that $\gamma^{-1}_{D_1}a_1= \gamma^{-1}_{D_2} a_2$.

\begin{claim} 
$\sim$ is an equivalence relation on $\bigcup_{x \in S}L_x$
which restricts to identity on each $L_x$.
\end{claim}
\begin{proof}
For $a_i \in L_{x_i}$ consider chains $C_i,D_i$ witnessing
 $a_i \sim a_{i+1}$ for $i=1,2$; that is $\gamma_{C_i}(a_i)
=\gamma_{D_i}(a_{i+1})$. By (ii) of Claim~\ref{6.3}
there are chains $E_1,E_2$ 
such that $\gamma_{E_1}\gamma_{D_1}(a_2)= \gamma_{E_2}\gamma_{C_2}(a_2)$
 whence $\gamma_{E_1} \gamma_{C_1}(a_1) =\gamma_{E_2}\gamma_{D_2}(a_3)$
and so $a_1\sim a_3$. If $x_1=x_2=x$ then $C_1=D_1=\{x\}$ 
and $\gamma_{C_1}=\gamma_{D_1}$ is identity of $L_x$ whence
$a_1=a_2$.
\end{proof}

We write $M=\left(\bigcup_{x\in S} L_x\right)$ and $L=M/\sim$
and denote the equivalence class of $a$ by $[a]$.
The following is immediate by (ii) of Claim~\ref{6.3} and its dual.
\begin{claim}
 For each $a \in M$
there are largest  $u$ resp. smallest $v$ in $S$
such that $L_u \cap [a] \neq \emptyset$ resp. $L_v \cap [a] \neq \emptyset.$
\end{claim}
We write $u=\lambda(a)$, $v=\mu(a)$ and observe that
for all $y$ with $\mu(a)\leq y \leq \lambda(a)$ 
one has unique $b =\tau_y(a)\in L_y\cap [a]$.
We also  write $\mu^*(a)$ for the unique $b \in L_{\mu(a)}$
such that $a \sim b$ and $\lambda^*(a)$ for the unique $c \in L_{\lambda(a)}$
such that $a \sim c$.

\subsection{Glueing of  posets}
Now, assume each   $L_x(x \in S)$ to be endowed with a partial order $\leq_x$,
 each with smallest element $0_x$
and greatest element $1_x$. Also, assume the $\gamma_{yx}$, $x\prec y$ in $S$,
to be 
 order isomorphisms $\gamma_{y.x}:[0_{y,x}, 1_x]_{L_x} \to [0_y, 1_{y,x}]_{L_y}$,
mapping an interval   of $L_x$ onto 
an interval of  $L_y$.
Moreover, it is required that $0_x <_x 0_{y,x}$ and
$1_{y,x}<_y 1_y$.

Again, observe that these conditions are satisfied if one replaces 
the order by its dual and the $\gamma_{y,x}$ by their inverses.

Observe that for $x\prec y \prec z$ 
the map $\gamma_{zx}=\gamma_{zy}\circ \gamma_{yx}$
is  either empty or
an order isomorphism $[0_{z,x},1_x]_{L_x}\to [0_z,1_{z,x}]_{L_z}$
where  $0_{z,x}=\gamma_{zx}^{-1}(0_z)$ and
$1_{z,x}=\gamma_{zx}(1_x)$.
Thus, for a maximal chain $C$ in $[x,y]$,
$\gamma_C$
is either empty or  an order isomorphism
$[0_{y,x}, 1_{x}]_{L_x} \to [0_y, 1_{y,x}]_{L_y}$.

For $a,b  \in M$   define $a\leq^0b$ if and only
 $a\in L_x$ and $b\in L_y$ for some  $x \leq y$ and if
 there are $x=x_0\prec x_1  \ldots \prec x_n=y$ 
and $a_i,b_i \in L_{x_i}$ $(0\leq i \leq n)$
such that $a=b_0\leq_{x_0}a_1$, $b_i=\gamma_{x_{i}x_{i-1}} a_i$ 
$(1\leq i \leq n)$, $b_i \leq_{x_{i}} a_{i+1}$ $(1\leq i <n)$,
and $a_{n+1}=b$. Here, put $a<^0b$ if $b_i <_{x_i}a_{i+1}$ for
some $0\leq i <n$.
Observe that $a \sim b$, otherwise. 
  Clearly, $\leq^0$ is transitive.

 Observe that applying the
above scheme to  $S'=[u,v]$ and $\bigcup_{x  \in S'} L_x$
one obtains the restriction of $\leq^0$ for this subset of $M$.
Thus, one may proceed by induction on the height of $S$.
In particular, with $a \leq^0 b$ witnessed as above,
one has $\gamma_{\one x_i} b_i \leq_\one \gamma_{\one x_i}a_{i+1}$. 
Now, for $a,b \in M$ define 
\[[a]\leq [b] \mbox{ if and only if there are } a'\sim a,\,
 b'\sim b \mbox{ such that } a'\leq^0 b'.\]

\begin{claim}\label{poset} \begin{itemize}
$(L,\leq )$ is a partially ordered set such that the following hold
for all $x,y \in S$ and $a,b \in M$.
\item[(i)]
For $a,b \in L_x$
 one has $a\leq_xb$ if and only if
$[a]\leq [b]$.
\item[(ii)] $x \mapsto [0_x]$ 
is an order embedding of $S$ into $L$
and 
$[0_x]\leq  [b]$ if and only if $x  \leq \lambda(b)$.
\item[(iii)] $[0_{x+y}]=\sup([0_x],[0_y])$ in $(L,\leq)$
for all $x,y \in S$.
\end{itemize} 
\end{claim}
\begin{proof} 
We claim that $a<^0 b$ implies $[a]\neq [b]$.
To prove this, we derive a 
contradiction from assuming 
 $a<^0b$ and $b\leq \gamma_{yx}a$. 
Proceeding by induction, it suffices to consider
$x=\zero$ and $y=\one$ in the definition of $<^0$. 
Thus, we have $0_{\one,\zero}\leq a=b_0\leq_\zero a_1$ and 
$\gamma_{\one,\zero} a_1 \geq \gamma_{\one,\zero} a\geq b$.
If $a=a_1$ then we apply induction for $S'=[x_1,\one]$.
If there is $j>0$ such that $b_j<_{x_j} a_{j+1}$
then $a_1<^0 b$ and we are done by induction.
Otherwise,
$b_j=a_{j+1}$ for all $j>0$, $a<_\zero a_1$, 
and $\gamma_{\one,\zero}a_1= b \leq_\one \gamma_{\one,\zero}a <_\one \gamma_{\one,\zero} a_1$ contradicting the requirement that
$\gamma_{\one,\zero}$ is an order isomorphism. 
It follows that $[a] \leq [b] \leq [a]$ 
implies $[a]=[b]$ proving that $\leq$ is a partial order on $L$.

To prove (i), consider $a<^0 b'\sim b$; one may assume $b'=\gamma_{zx}b$ with 
$z \geq x$ and, in view of induction, $x=\zero$ and $z=\one$. 
Thus, $\gamma_{z,\zero}b$ is defined for all $z$. 
Now choose $u\leq v$ in $S$ with $(u,v)$ minimal in $S^2$ 
such that there are $a'\in L_u$ and $b''\in L_v$ with
$\gamma_{u,\zero}a <_u a' \leq^0 b'' \sim b$.
By minimality of $v$ one has $u=v$ and so $a'\leq_u b''$.
 Since $\gamma^{-1}_{u,\zero} b''$ is defined, so is 
$\gamma^{-1}_{u,\zero} a'$ and it follows
$a= \gamma_{u,\zero}^{-1} \gamma_{u,\zero} a  <_\zero \gamma_{u,\zero} a'
\leq_\zero \gamma_{u,\zero}^{-1} b''=b$.  The converse is obvious.

The first claim in (ii) follows, immediately, from  
$0_x <_x 0_{y,x} \sim 0_y$ for $x \prec y$.
In $0_x \leq^0 b$  we may assume $b\in L_{\lambda(b)}$
and conclude $x \leq \lambda(b)$ from the definition of $\leq^0$.
The converse is obvious by this definition.

Ad (iii). By (ii) we have $b \geq^0 0_x, 0_y$
if and only if  $\lambda(b)\geq^0 x,y$, that is $\lambda(b) \geq x+y$,
which in turn is equivalent to $0_{\lambda(b)} \geq^0 0_{x+y}$.
Thus,  $b \geq^0 0_x, 0_y$ implies $b \geq^0 0_{x+y}$.
The converse is obvious.
\end{proof}

\subsection{Glueing of semilattices and lattices}
As an immediate  consequence  of Claims~\ref{join} and \ref{poset}
and of duality one obtains the following. 
\begin{lem}~\lab{glu} Consider a finite height modular lattice $S$
and a glueing $L$ of bounded posets $L_x$ $(x \in S)$.
\begin{itemize}
\item[(i)] If the $L_x$ are join semilattices then so is $L$.
Moreover, with the join operation $+_x$ on $L_x$
one has $\sigma(x)=[0_x]$ a join embedding $S \to L$
and joins in $L$ are computed according to Claim~\ref{join},
identifying $a\in L_x$ with $[a]$. 
\item[(iii)] If the $L_x$ are lattices then so is $L$.
Moreover, $L$ is an $S$-glued sum given by  $\sigma(x)=[0_x]$, $\pi(x)=[1_x]$
for $x\in S$. \item[(iii)] If $L$
is an $S$-glued sum of lattice $L_x$ given by $\sigma, \pi$
then $L$ is isomorphic to the glueing $L'$
of lattices $L'_x=L_x\times \{x\}$ with $0_x=(\sigma(x),x)$,
$1_x=(\pi(x),x)$, $0_{y,x}= (\sigma(y),x)$, $1_{y,x}=(\pi(x),y)$, and  
 glueing maps
$\gamma_{yx}((a,x))=(a,y)$ for $\sigma(y)\leq a \leq \gamma(x)$.  
\end{itemize}
\end{lem}

\subsection{Partial isomorphisms between glued sums.}

\begin{lem}\lab{gluesum}
Consider a lattice $L$
obtained by glueing $L_x$ $(x \in S)$
via $\gamma_{yx}$. 
 Let $U \subseteq S$ an antichain  and
assume that $L'$ is obtained by a glueing  $\phi_{yx}$ $(x\prec y$ in $S)$
 of the  lattices $L_x$ (having image $L'_x$ in $L'$) such that the following hold.
\begin{enumerate}
\item $\phi_{yx}=\gamma_{yx}$  if
$x\prec y$ in $S$ and  $\{x,y\}\cap U=\emptyset$. 
\item If $x\prec u \prec y$ and $u \in U$
then $\phi_{yu}(a)= \gamma_{yx}(\phi^{-1}_{ux}(a))$ for all 
$a \in [\phi_{ux}(0_{y,x}), 1_{u,x}]$.
\item If $x\prec u \prec y$ and $u \in U$
then  $\phi_{yu}(b)= \gamma_{yu}(b)$ for all $b\in [ 1_{u,x},1_u]$.
\end{enumerate}
Then the following hold.
\begin{itemize}
\item[(i)] There is an isomorphism $\chi:L_{|P} \to L'_{|P}$
such that $\chi \circ \gamma_{yx}= \phi_{yx} \circ \chi$
for all $(x,y) \in P$ where $P$ consists of the $(x,y)$, $x \prec  y$ in $S$ with $\{x,y\}\cap U=\emptyset$. 
\item[(ii)] 
If $T$ is an ideal of $S$ such that $U\cap T=\emptyset$ then $\chi$ restricts to an 
isomorphism of the ideal $L_T=\bigcup_{x\in T} L_x$
of $L$ onto the ideal $L'_T=\bigcup_{x \in T} L'_x$ of $L'$.
\end{itemize}
\end{lem}
According to (ii) we may identify $L_T$ with $L'_T$. Also, by (i) 
we may  use computations in $L_{|P}$ to verify
relations in $L'_{|P}$. 
\begin{proof}
Referring to (iii) of Lemma~\ref{glu}
define $\chi(a,x)=(a,x)$ for $a \in L_x$ and $x \not\in U$
to obtain a bijective map $\chi:L_{|P}\to L'_{|P}$
which restricts to an isomorphism $L_x \to L'_x$
for each $x \not\in U$ and
such that $\chi(0_{y,x}^L)= 0_{y,x}^{L'}$
and $\chi(1_{y,x}^L)= 1_{y,x}^{L'}$ 
if $(x,y) \in P$.
\end{proof}

\section{Basic models}\lab{6}

\subsection{The lattice directly associated with a group}
Adapting the construction in \cite[Section 4]{fm4},
fix a prime $p$.
Recall the abelian group $A$ and its
lattice  $L(A)$ from Subsection~\ref{skew}.
In particular, $A$ is a $\mathbb{Z}_{p^2}$-module.
The proper ``skeleton'' will be the ideal $S=L(pA)$ of $L(A)$.
.
\begin{fact}
$L(A)$ is an $S$-glued sum with embedding of  $S$ into $L(A)$
given by $\sigma(X)=X$ 
and $\pi(X)=\{ a \in A\mid pa \in X\}$.
The interval sublattices $[\sigma(X),\pi(X)]$ of $L(A)$
are isomorphic to lattices $L(\mathbb{Z}_p^4)$ and 
 $L(A)$ is a
finite  simple
modular lattice.
\end{fact}
\begin{proof}
$pA$ is a  free $\mathbb{Z}$-module  with generators
$f_i=pe_i$ and relations $pf_i=0$, $i=1,2,3$-  
Now, with $V=pA+\mathbb{Z}e_4$ one obtains an isomorphism
$f_i \mapsto e_i+V \in A/V$ 
and a lattice  isomorphism $\pi:S \to [V,A] \subseteq L(A)$
such that $\sigma(X) \leq pA \leq \pi(Y)$ for all $X,Y \in S$. 
It follows that $L=\bigcup_{x \in S}[\sigma (X), \pi(X]$ 
is a sublattice of $L(A)$.
Now, for any $C\in L(A)$  one has $X:=pC \in S $
and $C \subseteq \pi(X)$ whence $L=L(A)$.
Finally, $L(A)$ is simple in view of Lemma~\ref{simpglu}.
\end{proof}

Given a group $G$, let $Q$ denote the group ring  $\mathbb{Z}_{p^2}(G)$ 
with coefficients the integers modulo $p^2$.
The  free $Q$-module  $B$
with generators $e_1,e_2,e_3,e_4$ 
and relation $pe_4=0$ has a subgroup $A$ generated by $e_1,e_2,e_3,e_4$

\begin{fact}
  $L(A)$  embeds into the
$Q$-submodule lattice $L(B)$ of $B$ via $X \mapsto \varepsilon(X)=QX$;
moreover,  $\sigma'(X)= QX$ 
and $\pi'(X)= Q\pi(X)$ are lattice embeddings
of $S$ into $L(B)$ 
establishing an $S$-glued sum $L(G)$ within 
$L(B)$. Finally,
with the ideal $T=[0,\mathbb{Z}e_1+\mathbb{Z}e_3]$ of $S$,
  $L_T(G)=\bigcup_{X \in T}[QX,Q\pi(X)]$ is an ideal of $L(G)$.
\end{fact}
Observe that
$B$ and  $L(B)$ are finite
if so is $G$ and that  $L(G)$ is a ``rather small'' sublattice of $L(B)$. 

 \begin{proof}
Since $Q$ is a free $\mathbb{Z}_{p^2}$-module (with basis $G$),
$Q$ is a flat $\mathbb{Z}_{p^2}$-module  and
the map $X \mapsto QX =Q \otimes_{\mathbb{Z}_{p^2}} X$ 
is a lattice homomorphism $L(A) \to L(B)$
and injective since $L(A)$ is simple.
Thus, $\sigma'$ and $\pi'$ are lattice embeddings, too,
and $L(G)= \bigcup_{X \in S} [\sigma'(X), \pi'(X)]$
a sublattice of $L(B)$ which is an $S$-glued sum.
\end{proof}

\subsection{Capturing  a group generator  by a   stable term}

In $L_T(G)$ 
there is   a (canonical) skew $(3,2)$-frame $\Psi^0=(\Phi'^0,\Phi^0)$ given by the submodules
\[\Phi'^0:\;Qpe_1, Qpe_3, Qe_4, Q(pe_1-pe_3), Q(pe_1-e_4), Q(pe_3-e_4)\]
\[\Phi^0:\; Qe_1, Qe_3, Q(e_1-e_3). \]
Choose  $L_0(G)$ as the ideal  $[0, \top^{\Psi^0_{(3,2)}}]$
of $L_T(G)$.
The group $G$ embeds into
the group  of units of  the coordinate ring $R(\Phi'^0, L_0(G))$ via
\[ g \mapsto Q(pe_1 -gpe_3).    \] 
Recall the definitions of the derived skew frames 
$\Psi_{(3,2)}$ in Subsection~\ref{skew}
and the terms $g^*_1(\bar z',\bar z)$ and $g^+({\bar z}',\bar z)$
from  Lemma~\ref{stabpp}.

\begin{lem}\lab{basic}
For each  group $G$ and $g \in G$  
there is a modular lattice
$L(G,g)$  (finite if $G$ is finite)
with spanning  skew $(4,3)$-frame $\Psi=(\Phi',\Phi)$ 
of characteristic $p\times p$ 
and an isomorphism $\omega$ from $L_0(G)$ onto
the interval  
 $L_0(G,g)=[0, \top^{\Psi_{(3,2)}}]$ of $L(G,g)$
matching the skew $(3,2)$-frame $\Psi^0$ of $L_0(G)$ 
with $\Psi_{(3,2)}$,
inducing an isomorphism 
of coordinate rings
 $R(\Phi'^0,L_0(G))\to R(\Phi',L_0(G, g))$, and
such that  
\[(*)\quad\omega(Q(pe_1-gpe_3)=
g^+(\Psi)\]
which is a stable element w.r.t. the $3$-frame $\Phi'_{\neq 2}$.
\end{lem}
\begin{proof}
Leaving $(*)$ aside,
the lattice $L(G,g)$ and the isomorphism $\omega$ 
have been constructed in \cite[Section 4]{fm4}
relying on the method established in  Lemma~\ref{gluesum}.
Modularity follows from Lemma~\ref{modglu}.

In particular, we may assume $L_T(G)$ an ideal of $L(G,g)$
and $\omega$ identity. Moreover,
according to \cite[Lemma 18]{fm4}
one has  \[Q(pe_1- gpe_2)= g_1^*(\Psi) \in L_T(G)\]
with $g_1^*(\Psi)$ stable for $\Phi'$ according to \cite[Cor.13]{fm4}
and $(*)$ follows applying the perspectivity $\pi_{23}$
cf. Lemma~\ref{stabpp}.

Observe that in \cite{fm4} $Z_p$ has been used to denote both
the ring $\mathbb{Z}/p\mathbb{Z}$ and the ideal $pZ_{p^2}$
of $Z_{p^2}=\mathbb{Z}/p^2\mathbb{Z}$. Similarly, $R$ 
denoted both the ring $Q/pQ$ and the ideal $pQ$ of $Q$. 
Referring to the ideal, 
 given an element $a=\sum_{i=1}^3 r_i e_i \in A$ 
 one has the subgroup $Z_pa= \mathbb{Z}\sum_{i=1}^3 r_ipe_i$ of $pA$
and given $b=\sum_{i=1}^3 r_i e_i \in B$ 
one has the $Q$-submodule $Rb= Q\sum_{i=1}^3 r_ipe_i$
of $pB$. In each case, the proper meaning is obvious from  the context.
\end{proof}

\subsection{Basic model}

\begin{thm}\lab{const}
For each  group $G$ with generators $\bar g=(g_1, \ldots, g_n)$  
in $G$
there is a modular lattice
$L(G,\bar g)$  such that the following hold
\begin{enumerate}
\item $L(G,\bar g)$ contains an  $n$-tower $\Omega(n)_{can}$
(to be referred to as  \emph{canonical}) of skew $(4,3)$-frames  $\Psi^i=(\Phi'^i,\Phi^i)$  
of characteristic $p\times p$,  $i=1,\ldots, n$ 
\item There is an embedding $\gamma:G \to
R(\Phi'^n,L(G,\bar g))$ such that, for all $i$,
$g^+(\Psi^i)+\bot^{\Psi^n}
= \gamma(g_i)$
and $\gamma(g_i)\cdot  (a_1^{\Phi'^i}+a_3^{\Phi'^i})  =g^+(\Psi^i)$.
\item $L(G,\bar g)$ is finite if $G$ is finite.
\end{enumerate}
\end{thm}

\begin{proof}
Given $i$, consider  $L(G,g_i)$ from Lemma~\ref{basic}
with skew $(4,3)$-frame $\Psi^i=(\Phi'^i,\Phi^i)$ and 
isomorphism $\omega_i:L_0(G)\to L_0(G,g_i)$.
We may assume that the $L(G,g_i)$ are pairwise disjoint lattices.
Now, 
\[ \alpha_i(x)=  \omega_{i+1}(\omega_i^{-1}(x \cdot \top^{\Psi^i_{(3,2)}})) 
\in L_0(G,g_i) \]
defines an isomorphism  
\[\alpha_{i}:[a^{\Psi^i}_2, \top^{\Psi^i}]_{L(G,g_i)}
\to  [\bot^{\Psi^{i+1}}, \top^{\Psi^{i+1}_{(3,2)}}]_{L_{i+1}(G,g_{i+1})}.\] 
Let ${L}(G,\bar g)$  arise by 
Dilworth-Hall glueing 
(as described in Subsection~\ref{Dil}) the $L(G,g_i)$ 
via the isomorphisms $\alpha_{i}$. 
One has, due to the glueing via the $\alpha_i$, 
\[ (**)\; \omega_i(x)+ \bot^{\Psi^{j}}=\omega_j(x)
 \mbox{ for }
x \in L_0(G) 
\] 
for $j=i+1$. The case $i\leq j<n$
as well as the relations required for an $n$-tower follow by induction 
and transitivity of $\nearrow $ - recall $\Psi^i_{(3,2)} \nearrow
\Psi^{i (3,2)}$.
With the canonical  embedding 
\[\eta:G\to R(\Phi'^0,L(G,\bar g)) \mbox{ where } \eta(x)=
Q(pe_1-xpe_3)\]
equation
$(*)$ of Lemma~\ref{basic} together with $(**)$ for $j=n$ yield
\[\omega_n(\eta(g_i)) = \omega_i(\eta(g_i))+ \bot^{\Psi^n}=
   g^+(\Psi^i)+ \bot^{\Psi^n}.  \]
Since $G$ is generated by $\bar g$,
$\omega_n \circ \eta$ 
restricts to an embedding $\gamma:G \to R(\Phi'^n,L(G,\bar g))$
as required in (2).
\end{proof}

\section{Unsolvability}\lab{7}

In order to prove Theorem~\ref{thm}  applying Slobodkoi's Theorem~\ref{slo},
we show that there is an algorithm reducing the
Uniform Word Problem for the class $\mc{G}_0$ of all finite groups 
to the decision problem for the equational theory
of the class  $\mc{M}_0$ of finite modular lattices. 
 Moreover, we observe that
this reduction produces identities in $n+6$ lattice variables,
if applied to group presentations in  $n$ generators. 
To prove the reduction,  we verify the hypotheses of
 Lemma~\ref{reduction} which are given just before that lemma.

\begin{proof}
Consider 
a finite group presentation given by words $w_j(\bar g)$, $1\leq j \leq h$
in a list $\bar g=(g_1, \ldots, g_n)$
of generator symbols and  relations $w_j(\bar g)=e$, $j=1, \ldots, h$.
We construct a series of $n$-towers $\Omega_m$
of skew $(4,3)$-frames \[\Omega_m= (\Psi^k_m\mid_{ k=1, \ldots, n})
=(\Phi'^k_m,\Phi^k_m\mid_{ k=1, \ldots, n})\]\[
=(a'^k_{mi},c'^k_{mij};a^k_{mi},c^k_{mij} \mid_{k=1,\ldots ,n}), 0\leq m \leq \mu=n+h.\]
The list of generators of $\Omega_0$ 
is also denoted by $\bar a$, that of $\Omega_m$ by $\bar a_m$.
Let $F=F_0$ denote the modular lattice freely generated by the
$n$-tower $\Omega=\Omega_0$. The construction will be such that
$\Omega_m$ generates a sublattice $F_m$ of $F_0$ so that  $F_{m+1} \subseteq F_m$ for all $m<\mu$. 
Moreover, the following will  hold.
\begin{itemize}
\item[(A)] For the $4$-frame $\Phi'^n_\mu= (a'^n_{\mu, i}, c'^n_{\mu,ij})$ in $F_\mu$   
one has a list of elements $\bar s_\mu=(s_{\mu 1}, \ldots ,s_{\mu n})$ in
the group $R^ \#(\Phi'_\mu,F)$
such that $w_j(\bar s_\mu)=c'^n_{\mu,13}$ for $1\leq j\leq h$. 
\item[(B)] 
For any group $G$ and $\bar g$ in $G$
 with $w_j(\bar g)=e$ 
for $1\leq j \leq h$
 one has
 $\phi(\bar a)=\phi(\bar a_\mu)$ 
and $\phi(s_{\mu i})= \gamma(g_i)$ 
for $i=1, \ldots, n$ 
where $\phi:F \to L(G,\bar g)$ is
the homomorphism 
mapping $\bar a$ onto the canonical $n$-tower $\Omega_{can}$ of
the lattice 
$L(G,\bar g)$ constructed in Thm.~\ref{const}.
 \end{itemize}
Observe that $\phi$ in (B) exists by  Thms.~\ref{4tower} and  \ref{const} (1).
This construction will be uniform for 
all group presentations, to be implemented by an algorithm as
required in Lemma~\ref{reduction}.

In the context of this lemma, we consider quasi-identities
$\beta$ in the language of groups with 
 antecedent  $\alpha$ the conjunction 
 of identities $w_j(\bar y)=e$, $j=1,\ldots ,h$,
where $\bar y=(y_1,\ldots ,y_n)$. The presentation 
required in (a) of  Lemma~\ref{reduction} is that of an $n$-tower $\Omega$  of skew frames
- with generator symbols $\bar a$. 
Recall from Thm.~\ref{4tower} that $\Omega$ 
can be defined in terms of $n+6$ generators.

The terms $u_i(\bar x)$ are chosen such that
$\bar u(\bar a)$ is the $n$-tower generating $F_\mu$ within $F$.
Hypothesis $(i)$ is satisfied  due to Cor~\ref{4pro}.
Concerning hypothesis $(ii)$,
consider a homomorphism $\phi:F \to L\in \mc{M}$ and observe  that
 $ \phi(\bar u(\bar a))=\phi(\bar a_\mu) =\phi(\bar a)$ by (B).

The translation required in (b) 
is given by  the constant $c'^n_{13}$
defining the multiplicative identity  and the terms  defining  multiplication
and inversion in the group $R^\#(\Phi'^n,F)$
related to the $4$-frame $\Phi'^n=(a'^n_i,c'^n_{ij})$
which is part of the  $n$-tower $\bar a$.
According to Subsections 5.1 and 5.2 this translation 
satisfies hypothesis (iii) within $\mc{M}$. 
Also by this,   the algebra $G$  in (iv) is indeed a group, finite
if $L$ is finite. Moreover, the generators $\bar u|_n(\phi(\bar a))$ 
satisfy   $\alpha$ by $(A)$. 

Finally, hypothesis (v) is granted by    Theorem~\ref{const} and (B). 

\textbf{Outline of  construction:} 
To obtain $\bar a_\mu$ we put $\bar a_0=\bar a$ and construct, iteratively,
$n$-towers $\Omega_m=\bar a_m$, $m\leq \mu$.\\
\textbf{Case 1}: $m\leq n$ 
\begin{enumerate}
\item The $n$-tower $\Omega_m=\bar a_{m}$ is obtained from 
the $n$-tower $\Omega_{m-1} =\bar a_{m-1}$
by lower reduction, induced by a reduction of $\Psi^{m-1}_{m-1}$
to $\Psi^m_m$,
within  the sublattice  $F_{m-1}$ of $F$ generated by
 $\bar a_{m-1}$,
\item  
One has $m$  elements $s_{m1}, \ldots ,s_{mm}$
$3$-stable in $F$ for the $4$-frame $\Phi'^n_m$ in $\Omega_m$.
\item  $s_{mi}$ $(i \leq m-1)$ is obtained as in  Fact~\ref{stab}
 by the
lower reduction in (1)    from $s_{m-1,i}$ 
stable in $F$ for $\Phi'^n_{m-1}$ while 
 $s_{mm} =s+\bot^{\Psi_m^n}$ where $s=g^+(\Psi_m^m)$ is 
 $3$-stable in $F$ for the $4$-frame   $\Phi_m'^m$.
\item The reduction in (3) is chosen  such that
the skew frame $\Psi^m_{m-1}$
is reduced
as in Lemma~\ref{charp} to the skew frame $\Psi^m_m$ having characteristic $p\times p$.
\end{enumerate}
\textbf{Case 2}: $n<m \leq \mu=n+h$.
\begin{enumerate}\setcounter{enumi}{4}
\item
 $\bar a_{m}$ is obtained from $\bar a_{m-1}$
by  upper reduction 
within  the sublattice  $F_{m-1}$ of $F$ generated by
 $\bar a_{m-1}$.
\item  
One has a list $\bar s_m$ of $n$  elements 
stable in $F$ for the $4$-frame $\Phi'^n_m$
and 
satisfying $w_j(\bar s_{m})= c'^n_{m,13}$ for $j \leq m-n$,
within the group  $R^\#(\Phi'^n_m,F)$.
\item These are obtained from $\bar s_{m-1}$ by the upper reduction  in (5).
\end{enumerate}

Proof of (A) and  (B). 
We show, by induction, that for all $m \leq \mu$
 \begin{itemize}
\item  $\phi(\bar a_m)= \phi(\bar a)$, that is $\phi(\Omega_m)=\phi(\Omega)$.
\item $\bar s_{m}$ 
is a list of stable elements for $\Phi'^n_m$ 
\item $\phi(s_{mi})=\gamma(g_i)$ for $i\leq \min(m,n)$.
\item $w_j(\bar s_{m})= c'^n_{m,13}$ where $m\geq n$ and $j\leq h= m-n$. 
\end{itemize}
 The case $m=0$ is just the definition
of $\phi$.  
For  $m\leq n$, we apply  Lemma~\ref{charp}
to $\Psi_{m-1}^m$, that is with ${\bf b}=b^*(\bar a'^m_{m-1},\bar a^m_{m-1})$ and
${\bf d}=d^*(\bar a'^m_{m-1},\bar a^m_{m-1})$. 
By inductive hypothesis one has $\phi(\Psi^m_{m-1})=\phi(\Psi^m)$
which
is of characteristic $p\times p$
as part of the canonical 
 $n$-tower of $L(G,\bar g)$, whence  $\phi({\bf b})= \phi(\bot^{\Psi^m_m})$
and $\phi({\bf d})= \phi(a_1'^{\Psi^m_m})$  in view of (4). It follows
 \[\phi(\Psi^m_m) =\phi((\Psi^m_{m-1})_{{\bf b},{\bf d}})=
(\phi(\Psi^m_{m-1}))_{\phi({\bf b}),\phi({\bf d})}=  
\phi(\Psi^m_{m-1})=\phi(\Psi^m).\]
This in turn implies $\Phi(\Omega_m)= \Phi(\Omega_{m-1})$
in view of Observation~\ref{towred}.      
The  element $s$ in (3) being  chosen 
as $s=g^+(\Psi^m_m)$ according to Lemma~\ref{stabpp}
we have $s$ $3$-stable for $\Phi'^m_m$ and  $s_{mm}$ $3$-stable for $\Phi'^n_m$ due to the isomorphism induced by
$(\Psi^m_m)_{3,2} \nearrow (\Psi^n_m)_{3,2}$ 
which matches $\Phi'^m_m$ with $\Phi'^n_m$. 
Moreover, 
 $\phi(s_{mm})=\gamma(g_m)$ by (2) of Thm.~\ref{const}.
For $i<m$ the other hand, according to (3) 
 $s_{mi}$  is obtained from $s_{m-1,i}$ as in Fact~\ref{stab}
applying the reduction with ${\bf b}_1+a^n_\bot$ 
to the $4$-frame $\Phi'^n_{m-1}$. In particular, $s_{mi}$
is stable for $\Phi'^n_m$.
Since $ \phi(\Omega_m ) =\phi(\Omega_{m-1})$ 
  it follows that 
 $\phi(s_{mi})=\phi(s_{m-1,i})=\gamma(g_i)$.
   
For $m=n+j$ we proceed with the same kind of reasoning,
now referring to Lemma~\ref{stabgp}, 
to add  $w_j(\bar s_m)=c'^n_{m,13}$, while
stable $\bar s_{m-1}$ leads to  
 stable $\bar s_m$, 
and   $\phi(s_{m-1,i})=\gamma(g_i)$ 
to $\phi(s_{m,i})=\gamma(g_i)$ and
$w_i(\bar s_{m-1})=c'^n_{m-1,13}$ to  $w_i(\bar s_m)=c'^n_{m,13}$
for all $i<m-n$.

\end{proof}

\section{Remarks}\lab{5.4}
Reducing (Restricted)  Word Problems for groups
to such for modular lattices follows the same scheme
in the finite and in the infinite case.
Recall from Subsection~\ref{coor}  
that with any $4$ frame in a modular lattice
one has the associated von Neumann coordinate  ring $R(\Phi)$
with subgroup $R^*(\Phi)$ of units,
all defined in terms of the frame.

Now, with a group presentation $(\Pi,\bar g)$ 
associate the lattice presentation $\lambda(\Pi,\bar g)$
obtained as follows:  To the  $4$-frame $\Phi$
  add a  generator symbol ${g}_i$
for each $g_i$; also, to the relations add
the relations  $a_1{g}_i=a_\bot$, $a_1+ {g}_i= a_1+a_2$,
and   $w_i({\bar g})=c_{13}$. Here. $w_i(\bar g)=e$
is a relation of $(\Pi,\bar g)$.

By  this construction,  if $w(\bar g)=e$ 
is a consequence of $(\Pi,\bar g)$ 
for (finite) groups, then $w({\bar g}) =c_{13}$
is a consequence of $\lambda(\Pi,\bar g)$
for (finite)  modular lattices $L$, since
 $R^*(\Phi,L)$ is a (finite) group for any $L$.

On the other hand, if $(G,\bar h)$ is a model of 
$(\Pi,\bar g)$ such that $w(\bar h) \neq e$  (and $G$ finite),
then a (finite) model $(L,\Phi,{\bar h})$
with  $w({\bar h}) \neq c_{13}$ is obtained
choosing a (finite)
  vector space $_FV$ 
of  $\dim_FV= 4|G|$; then the lattice
 $L$  of $R$-submodules of $R^4$, $R$ the group ring $F[G]$,
  embeds into the lattice $L(_FV)$  of subspaces; moreover,
the  canonical $4$-frame $\Phi$ of $L$ together
with the $R(e_1- {h}_ie_3) \in R^*(\Phi)$ provide the required model
 of $\lambda(\Pi,\bar g)$ such that $w({\bar g})\neq c_{13}$.
In particular,
the model embeds into the subspace lattice $L(_FV)$.
Consequently, the relevant class of models
consists of sublattices of $L(_FV)$ where $\dim_FV$ 
is infinite respectively of $L(_{F_d}V_d)$ where $\dim_{F_d}V_d \to \infty$.
In particular, from Thm.~\ref{slo} one has the following.
\begin{cor}
Let $\mc{C}$ 
be any class of finite modular lattices 
such that for all $d \in \mathbb{N}$ 
there are lattices of subspaces $L(_{F_d}V_d)$ in $\mc{C}$ 
with $\dim_{F_d}V_d \geq d$. Then the Restricted Word Problem
for $\mc{C}$ is unsolvable. 
\end{cor}

With a simple modification one can restrict the number of lattice
generators  to $5$: For any $n\geq 3$
$n$-frames have an equivalent presentation   in modular
lattice theory  to a presentation with  $4$-generators \cite[Satz 4.1]{quad}
and with $g=\sum_{j=1}^n \pi_{3,j+3}(g_j)$
one obtains $g_j= \pi_{j+3,3}(g\cdot (a_1+a_{j+3}))$ 
for $j=1, \ldots ,n$ to replace $\bar g$ equivalently by $g$
and to proceed with the $4$-frame given by $a_i,c_{ij}$
where $i,j\leq 4$.

For reduction  to identities, the scheme is modified 
as described in Subsection~\ref{2.6}. 
In order 
to associate with group generators lattice terms 
which allows one to force group  relations within the lattice,
several frames are combined via some kind of glueing.
This leads to models which 
are  non-Arguesian lattices and, in particular,
do not embed into lattices of normal subgroups.

In all examples, discussed, one has a certain set
$\Sigma$ of quasi-identities in the language of groups
and for each $\beta \in \Sigma$
an associated quasi-identity $\lambda(\beta)$ in the language
of lattices  and a class $\mc{S}$ of (finite) modular lattices,
the class of ``models'', such that the following hold.
\begin{itemize}
\item If $\beta$ holds for all (finite) groups then
$\lambda(\beta)$ holds for all (finite) modular lattices. 
\item If $\beta$ fails for some (finite) group then
$\lambda(\beta)$ fails for some ``model''  lattice  in $\mc{S}$. 
\end{itemize} 
Thus, if $\Sigma$ is undecidable for the class of (finite)
groups, then the set of $\lambda(\beta)$ valid in
all (finite) modular lattices and the set of
$\lambda(\beta)$ failing in some lattice in $\mc{S}$
are recursively inseparable. In other words,
the undecidability results extend to all
classes of (finite) modular lattices
containing the relevant class of models.

Observe that the number of generators in Slobodkoi's 
Theorem is $3m+61$ where $m$ is the minimum number of states
of a two tape Minsky machine computing some 
partial recursive function with non-recursive domain.

\begin{prob}
What is the minimal $N$ such that 
the $N$-variable equational theory of 
finite modular lattices is undecidable.
\end{prob}
Since skew $(n,m)$-frames can be generated by
$8$ elements, the following could be of use.
\begin{prob}
Can one find $n-m$ stable elements  in the modular
lattice freely generated by a skew $(n,m)$-frame
of characteristic $p\times p$?
\end{prob}

\section{Funding and Declarations}
\subsection{Funding} None
\subsection{Declarations} Not applicable
\subsection{Acknowledgements} The author
is deeply obliged to the  referee for
thoroughly reading and correcting and for suggestions,
improving the presentation, substantially.

\end{document}